\documentclass[12pt,reqno]{amsart}

\usepackage{amsfonts}
\usepackage{amsmath}
\usepackage{amssymb, wasysym}
\usepackage{amsthm}
\usepackage{xcolor}
\usepackage{graphicx}

\usepackage{tikz}
\usepackage{pgfplots}
\pgfplotsset{compat=1.18}

\usepackage{hyperref}
\hypersetup{
	colorlinks=true,
	linkcolor=blue,
	citecolor=blue,
	urlcolor=cyan
}

\allowdisplaybreaks
\usepackage{tikz}
\newtheorem{main}{}[section]

\newtheorem{theorem}[main]{Theorem}

\numberwithin{equation}{section} % Makes it 1.1, 1.2, etc.

\newtheorem{corollary}[main]{Corollary}

\makeatletter
\let\c@equation\c@main % Link equation counter to main counter
\makeatother
\allowdisplaybreaks
\begin{document}
	\title{Radii of Starlikeness of Spirallike Functions}
	\author[A. Lakshmipriya and V. Ravichandran]{ Lakshmipriya Arul \and Vaithiyanathan Ravichandran}
	
	\address{Department of Mathematics,
		National Institute of Technology,
		Tiruchirappalli-620015, India}
	
	\email{arullakshmipriya24@gmail.com; ravic@nitt.edu, 		vravi68@gmail.com}

	\begin{abstract}A normalized analytic function $f$ defined on the unit disk $\mathbb{D}$ is called $\alpha$-spirallike if $f(\mathbb{D})$ is an $\alpha$-spirallike domain, that is, for every $z\in f(\mathbb{D})$, the $\alpha$-spiral $ze^{-e^{i\alpha}t}$, $t\ge0$, is contained in $f(\mathbb{D})$, or equivalently, if $\operatorname{Re}\!\left(e^{i\alpha} {zf'(z)}/{f(z)}\right)>0$ for all  $z\in \mathbb{D}$. The 0-spirallike functions are the usual starlike functions. The function $f$ is parabolic starlike, if  $\operatorname{Re}\!\left(zf'(z)/f(z)\right)
		>
		\left|(zf'(z)/f(z))-1\right|$, lemniscate starlike  if  $|(zf'(z)/f(z))^2-1|<1$ and lune starlike if $|(zf'(z)/f(z))^2-1|<2|zf'(z)/f(z)|$. In this paper, we determine  the radii of parabolic starlikeness, lemniscate starlikeness and lune starlikeness of $\alpha$-spirallike  functions. Our approach combines geometric characterizations of the underlying regions with the resultant method for eliminating auxiliary variables, leading to explicit equations for the sharp radii.  Some known radius results are recovered as special cases.
	\end{abstract}
	\maketitle
	\section{Introduction}Let \(\mathbb{D}\) be the unit disk consisting of points \(z\in \mathbb{C}\) with \(|z|<1\). The class \(\mathcal{A}\) consists of analytic functions \(f:\mathbb{D}\to \mathbb{C}\) normalized by \(f(0)=0\) and \(f'(0)=1\). The subclass \(\mathcal{S}\) comprises all univalent functions in \(\mathcal{A}\). Let \(\mathcal{P}\) denote the class of analytic functions \(p:\mathbb{D}\to\mathbb{C}\) satisfying \(p(0)=1\) and \(\operatorname{Re}(p(z))>0\) for all \(z\in\mathbb{D}\). Functions in \(\mathcal{P}\) are called Carath\'{e}odory functions or simply functions with positive real part.
	A domain \(\Omega\) is called \(\alpha\)-spirallike if, for every \(z\in\Omega\), the \(\alpha\)-spiral
	\(ze^{-e^{i\alpha}t}\), \(t\ge0\), lies entirely in \(\Omega\). A function \(f\in \mathcal{A}\) is called \(\alpha\)-spirallike if \(f(\mathbb{D})\) is an \(\alpha\)-spirallike domain, and the class of all such functions is denoted by \(\mathcal{SP}_{\alpha}\). Equivalently, \(f\in \mathcal{SP}_{\alpha}\) if and only if \[\operatorname{Re}\!\left(e^{i\alpha}\frac{zf'(z)}{f(z)}\right)>0,\quad  z\in\mathbb{D}.\] Moreover, \v{S}pa\v{c}ek \cite{Zbl0006.06403} proved that every function in the class \(\mathcal{SP}_{\alpha}\) is univalent in \(\mathbb{D}\). For \(\alpha=0\), the class \(\mathcal{SP}_0\) is the well-known class \(\mathcal{ST}\) of starlike functions.
	
	We shall be interested in three specific subclasses of starlike functions. The first is the class of parabolic starlike functions introduced by R{\o}nning \cite{MR1128729} in his investigation of uniformly convex functions following the definition of uniformly starlike and uniformly convex function by Goodman \cite{MR1097287, MR1145573}.  A function \(f\in\mathcal{A}\) is called uniformly convex if it maps every circular arc in \(\mathbb{D}\), whose center also lies in \(\mathbb{D}\), onto a convex arc. Analytically, a function \(f\in\mathcal{A}\) is uniformly convex if \(1+zf''(z)/f'(z)\) takes values in the parabolic region \(\Omega_{par}:=\{w\in\mathbb{C}:\operatorname{Re}w>|w-1|\}\). These two classes are connected through the Alexander relation \cite{MR1503516}: \(f\) is uniformly convex if and only if \(zf'\) is parabolic starlike. Thus, the  function \(f\in\mathcal{A}\) is   parabolic starlike if \(zf'(z)/f(z)\) takes values in the parabolic region \(\Omega_{par}\).  In 1996, Sok{\'o}{\l} and Stankiewicz \cite{MR1473947} introduced a subclass of \(\mathcal{A}\) consisting of functions starlike with respect to the lemniscate of Bernoulli, \( \Omega_{L}:=\{w\in\mathbb{C}:|w^2-1|<1\}. \) Therefore, a function \(f\in\mathcal{A}\) is lemniscate starlike if \(zf'(z)/f(z)\) takes values in \(\Omega_{L}\).  The third is the class of lune starlike functions introduced by Raina and Sok\'o{\l} \cite{MR3419845}, which is associated with the lune region \(\Omega_{\!\!\rightmoon}:=\{w\in\mathbb{C}:|w^2-1|<2|w|\}.\) Let \(W(z)=zf'(z)/f(z),\) then we define the following:

	The \emph{radius of parabolic starlikeness} of the class \(\mathcal{SP}_{\alpha}\) is defined by
	\[
	R_{\mathcal S_{par}}(\mathcal{SP}_{\alpha})
	=
	\sup\left\{
	r\in(0,1]:
	\operatorname{Re}\left(W(z)\right)>
	\left|W(z)-1\right|
	\text{ for all }|z|<r,\ f\in\mathcal{SP}_{\alpha}
	\right\}.
	\]
	Similarly, the \emph{radius of lemniscate starlikeness} of the class \(\mathcal{SP}_{\alpha}\) is defined by
	\[
	R_{\mathcal S_{L}}(\mathcal{SP}_{\alpha})
	=
	\sup\left\{
	r\in(0,1]:
	\left|\left(W(z)\right)^2-1\right|<1
	\text{ for all }|z|<r,\ f\in\mathcal{SP}_{\alpha}
	\right\}.
	\]
	Finally, the \emph{radius of lune starlikeness} of the class \(\mathcal{SP}_{\alpha}\) is defined by 
	\[
	R_{\mathcal S_{\!\!\rightmoon}}(\mathcal{SP}_{\alpha})
	=
	\sup\left\{
	r\in(0,1]:
	\left|\left(W(z)\right)^2-1\right|<2\left| W(z) \right|
	\text{ for all }|z|<r,\ f\in\mathcal{SP}_{\alpha}
	\right\}.
	\]
	
	We determine these radii for the class \(\mathcal{SP}_{\alpha}\) by analyzing the tangency of the boundary curves of the corresponding regions. The well-known radius results for the class \(\mathcal{ST}\) associated with the parabolic, lemniscate and lune regions \cite{MR1415180, MR3902942, MR4861252} are recovered as special cases of our results when \(\alpha=0.\)
	
	\section{Preliminaries}
	Before presenting the main results, we briefly recall the construction and a fundamental property of the resultant. Resultants provide an algebraic method for eliminating variables from systems of polynomial equations. Consider the system
	\[
	\begin{cases}
		f(x,y,z)=0,\\
		g(x,y,z)=0,\\
		h(x,y,z)=0.
	\end{cases}
	\]
	To eliminate the variable \(z\), one may compute the resultants of suitable pairs of polynomials, treating them as polynomials in \(z\) whose coefficients depend on \(x\) and \(y\). For instance, let
	\[
	\begin{aligned}
		f(z)&=a_nz^n+a_{n-1}z^{n-1}+\cdots+a_0,\\
		g(z)&=b_mz^m+b_{m-1}z^{m-1}+\cdots+b_0,
	\end{aligned}
	\qquad n>m,
	\]
	where \(a_i\) and \(b_i\) are functions of \(x\) and \(y\). The resultant \(\operatorname{Res}_z(f,g)\) is defined as the determinant of the Sylvester matrix \cite{10.1098/rspl.1850.0120} of \(f\) and \(g\), which is a square matrix of order \(n+m\). This matrix is constructed by arranging the coefficients of \(f(z)\), in descending powers of \(z\), in the first \(m\) rows, and those of \(g(z)\) in the next \(n\) rows. Each successive row is obtained by shifting the coefficients one position to the right and filling the remaining positions with zeros. The resultant is a polynomial in \(x\) and \(y\), and its vanishing is a necessary condition for \(f\) and \(g\) to have a common root in \(z\) \cite{salmon1859lessons}.
	
	\begin{equation*}
		\operatorname{Res}_z(f,g)=
		\begin{vmatrix}
			a_n & a_{n-1} & \cdots & a_0 & 0 & \cdots & 0\\
			0 & a_n & a_{n-1} & \cdots & a_0 & \cdots & 0\\
			\vdots & & \ddots & & & \ddots & \vdots\\
			0 & \cdots & 0 & a_n & a_{n-1} & \cdots & a_0\\[2pt]
			b_m & b_{m-1} & \cdots & b_0 & 0 & \cdots & 0\\
			0 & b_m & b_{m-1} & \cdots & b_0 & \cdots & 0\\
			\vdots & & \ddots & & & \ddots & \vdots\\
			0 & \cdots & 0 & b_m & b_{m-1} & \cdots & b_0
		\end{vmatrix}_{(n+m)\times(n+m)}.
	\end{equation*}
	Let \(R_1=\operatorname{Res}_z(f,g)\) and \(R_2=\operatorname{Res}_z(g,h)\). The vanishing of \(R_1\) indicates that \(f\) and \(g\) have a common root in \(z\), while the vanishing of \(R_2\) indicates that \(g\) and \(h\) have a common root in \(z\). Treating \(R_1\) and \(R_2\) as polynomials in \(y\), with coefficients depending on \(x\), we eliminate \(y\) by computing the resultant
	\[
	R(x)=\operatorname{Res}_y(R_1,R_2),
	\]
	which is a polynomial in \(x\). Every solution \((x,y,z)\) of the original system necessarily yields \(R(x)=0\), under the usual assumptions on the polynomial degrees. However, the converse need not hold: a root of \(R(x)=0\) may arise from the elimination process without corresponding to a solution of the original system. Therefore, the roots of \(R(x)=0\) provide candidate values of \(x\), which must be checked against the original system. Among the solutions produced by the candidate values of \(x,\) we call a solution \((x,y,z)\) \textit{admissible} if it satisfies all three equations of the original system. Solutions of the eliminated equations that do not satisfy the original system are termed \textit{extraneous solutions}.
	
	We intend to find the radius of starlikeness with respect to several domains.  Therefore, we want to find the variability region of \(zf'(z)/f(z)\) for \(f\in\mathcal{SP}_{\alpha}.\)  For every \(f\in\mathcal{SP}_{\alpha}\), there exists a function \(p\in\mathcal{P},\) \cite{MR704183} satisfying
	\begin{equation*}\label{speqn}
		e^{i\alpha}\frac{zf'(z)}{f(z)}
		=
		p(z)\cos\alpha+i\sin\alpha.
	\end{equation*}
	Let $
	q(z)= {zf'(z)}/{f(z)}
	$
	so that
	\begin{equation*}\label{qexp}
		q(z)
		=
		e^{-i\alpha}\bigl(p(z)\cos\alpha+i\sin\alpha\bigr),
	\end{equation*}
	from which \(p(z)\) can be written as
	\begin{equation*}\label{pdefn}
		p(z)=\frac{e^{i\alpha}q(z)-i\sin\alpha}{\cos\alpha}.
	\end{equation*}
	Since \(p\in\mathcal P\), the classical Carath\'eodory estimate \cite{ali2026introduction} yields
	\[
	\left|
	p(z)-\frac{1+r^2}{1-r^2}
	\right|
	\le
	\frac{2r}{1-r^2},
	\qquad |z|=r<1.
	\]
	From this estimate, it can be easily derived that 
	\begin{equation}\label{p1reqdisk}
		|q(z)-(u(r)+iv(r))|\le R(r),
	\end{equation}
	where the center and radius are given by
	\[
	u(r)=\frac{1+r^2\cos2\alpha}{1-r^2},
	\qquad
	v(r)=-\frac{r^2\sin2\alpha}{1-r^2},
	\qquad
	R(r)=\frac{2r\cos\alpha}{1-r^2}.
	\]
	The disk \eqref{p1reqdisk} describes the exact variability region of
	\(q(z)=zf'(z)/f(z)\) for \(f\in\mathcal{SP}_{\alpha}\) and \(|z|\le r\).
	Indeed, for the function \(f=f_0\) such that
	\[
	\frac{zf_0'(z)}{f_0(z)}
	=
	e^{-i\alpha}
	\left(
	\cos\alpha\,\frac{1+z}{1-z}
	+i\sin\alpha
	\right),
	\]
	the boundary of the disk \eqref{p1reqdisk} is attained. Consequently, the image of \(|z|\le r\) under \(q(z)=zf_0'(z)/f_0(z)\) is precisely the disk \eqref{p1reqdisk}. 	%	For any function \(f\in \mathcal{SP}_{\alpha}\), there exists a \(p\in \mathcal{P}\) such that
	The boundary of the disk \eqref{p1reqdisk}, given by
	\(
	|q(z)-(u(r)+iv(r))|=R(r)
	\)
	admits the parametrization
	\[
	q(z)-(u(r)+iv(r))=R(r)e^{it},
	\qquad 0\le t<2\pi .
	\]
	We define \(\gamma\) to be the boundary of the disk,  therefore
	\begin{equation*}
		\gamma(t)
		:=
		\bigl(u(r)+R(r)\cos t\bigr)
		+
		i\bigl(v(r)+R(r)\sin t\bigr).
	\end{equation*}
	If we define \(x(r,t)=u(r)+R(r)\cos t \text{ and }
	y(r,t)=v(r)+R(r)\sin t,\) then the above equation can be written as 
	\begin{equation}\label{p1gammaeqn}
		\gamma(t)=x(r,t)+iy(r,t).
	\end{equation}
	Let \(\sin t=s, \cos t=c\), then we have \(
	x=u+Rc,
	\qquad
	y=v+Rs.
	\)

	\section{Radius of parabolic starlikeness}  We now determine the radius of parabolic starlikeness of the class \(\mathcal{SP}_{\alpha}\).  Since  the  variability region of \(zf'(z)/f(z)\) is the disk \eqref{p1reqdisk}, for  the proof of our next theorem,  it suffices to determine the largest \(r\) for which this disk is contained in the parabolic region \(\Omega_{par}\).

	\begin{theorem}Let  $0<|\alpha|<\pi/2$.  The radius of parabolic starlikeness for the class
		$\mathcal{SP}_\alpha$  is the smallest admissible root in $(0,1)$ of
		\begin{align}\label{p1redeq}
			T_{par}(r,\alpha)& :=(\cos8\alpha-8\cos6\alpha+4\cos4\alpha+40\cos2\alpha+35)r^6\\ \notag
			&\quad{} +(14\cos6\alpha+24\cos4\alpha-6\cos2\alpha-40 )r^4 \\ \notag
			&\quad{}  -(12\cos4\alpha+12\cos2\alpha-24)r^2+(4\cos 2\alpha -4) =0.
		\end{align}
		This radius is sharp.
	\end{theorem}

	\begin{proof}

		Let
		\(
		\Omega_{par}=\left\{w\in\mathbb{C}:|w-1|<\operatorname{Re}(w)\right\}.
		\)  Then the boundary of \(\Omega_{par}\) is characterized by
		\(
		|w-1|=\operatorname{Re}(w).
		\)
		Writing \(w=x+iy\), the boundary equation becomes
		\(
		(x-1)^2+y^2=x^2,
		\)
		which simplifies to
		\[
		\Phi(x,y):=y^2-2x+1=0.
		\]
		Furthermore, we observe that
		\[
		\nabla\Phi(x,y)=(-2,\,2y)\neq (0,0)
		\]
		at every point in the domain which confirms that \(\Phi\) determines a smooth curve.  We define the function $F: (0,1) \times [0,2\pi] \to \mathbb{R}$ along
		the curve $\gamma(t) $ in \eqref{p1gammaeqn} by
		\begin{equation*}
			F(r,t) := \Phi(\gamma(t)) = y^2(r,t) - 2(x(r,t)) + 1.
		\end{equation*}
		The sign of \(F(r,t)\) determines the position of the point \(\gamma(t)\) relative to the region \(\Omega_{par}\). More precisely,
		\begin{enumerate}
			\item \(F(r,t)=0\) if and only if \(\gamma(t)\) lies on the boundary of \(\Omega_{par}\),
			\item \(F(r,t)<0\) if and only if \(\gamma(t)\in\Omega_{par}\),
			\item \(F(r,t)>0\) if and only if \(\gamma(t)\notin\overline{\Omega_{par}}\).
		\end{enumerate}
		Consequently, the disk determined by \(\gamma(t)\) is contained in \(\Omega_{par}\) only when
		\[
		F(r,t)\le 0
		\qquad \text{for all } t.
		\]
		The extremal situation occurs when equality is attained for at least one value of \(t\).

		\textbf{Tangency condition:} At the extremal radius, the boundary circle \(\gamma\) is tangent to the parabola defining \(\partial\Omega_{par}\). As a result, there exists \(t=t_0\) such that
		\begin{align}\label{p1cond1}
			F(r,t)&=0\\
			\label{p1cond2}
			\frac{\partial F(r,t)}{\partial t}&=0.
		\end{align}
		From a geometric perspective, \(\nabla\Phi\bigl(\gamma(t_0)\bigr)\) represents a normal vector to the boundary \(\partial\Omega_{par}\), whereas \(\gamma'(t_0)\) is the tangent vector to the boundary circle at the point of contact. As a result, the tangency condition is defined by the orthogonality relation
		\(
		\nabla\Phi(\gamma(t_0))\cdot\gamma'(t_0)=0,
		\)
		which is equivalent to
		\[
		\frac{\partial F}{\partial t}(r,t_0)=0.
		\]
		Evaluating \eqref{p1cond1} and \eqref{p1cond2} using \eqref{p1gammaeqn},  we obtain
		\begin{align*}
			F(r,t)=&(v+Rs)^2-2(u+Rc)+1=0,\\
			F_t(r,t)=&2(v+Rs)Rc+2Rs=0.
		\end{align*}
		Substituting the expressions for \(u(r)\), \(v(r)\), and \(R(r)\) given in \eqref{p1reqdisk} and multiplying throughout by \((1-r^2)^2\) to eliminate denominators and applying the identity \(s^2=1-c^2\), we have
		\begin{align}\label{p1f1andf2}
			F_1&:=(1-r^2)^2F(r,t)
			={}g_0+g_1s=0\\ \notag
			F_2&:=(1-r^2)^2F_t(r,t)=h_0+h_1s
			=0.
		\end{align}
		where \begin{align*}
			g_0&=
			r^4\sin^22\alpha
			-2(1-r^2)(1+r^2\cos2\alpha)
			+(1-r^2)^2
			\\
			&\quad {}+4r^2\cos^2\alpha
			-4r\cos\alpha(1-r^2)c-4r^2\,c^2\cos^2\alpha
			\\
			g_1&={}-4r^3\cos\alpha\sin2\alpha \\
			h_0&={}-4r^3\,c\cos\alpha\sin2\alpha \\
			h_1&=4r\cos\alpha\bigl((1-r^2)+2r\,c \cos\alpha\bigr)
		\end{align*}
		Solving \(F_2=0\) for \(s\), we have
		\begin{equation}\label{p1sval}
			s=
			\frac{r^2\,c \sin2\alpha}
			{1-r^2+2r\,c \cos\alpha} \qquad \left(0<|\alpha|<\frac{\pi}{2}\right)
		\end{equation}
		If the denominator vanishes, then it follows from \eqref{p1f1andf2} that
		\[
		F_2={}-4r^3\,c\cos\alpha\sin2\alpha=0.
		\]
		Solving this equation together with \(F_1=0\) for $r,$ yields solutions satisfying \(r>	R_{\mathcal S_{par}}(\mathcal{SP}_{\alpha})\). Since \(	R_{\mathcal S_{par}}(\mathcal{SP}_{\alpha})\) is the smallest admissible radius for the system \eqref{p1f1andf2} at which tangency occurs, every other solution corresponds to a tangency that takes place after the disk has already ceased to be contained in \(\Omega_{par}\). Consequently, the case
		\(
		1-r^2+2r\,c \cos\alpha=0
		\)
		may be excluded from further consideration.
		Substituting \eqref{p1sval} into the identity \(s^2=1-c^2\), yields
		\begin{equation} \label{p1geqn}
			G(c)	=b_4c^4+b_3c^3+b_2c^2+b_1c+b_0
		\end{equation}
		where the coeffecients \(b_i\) are given by
		\begin{equation*}
			\begin{aligned}
				b_4&=4r^2\cos^2\alpha, &
				b_3&=4r\cos\alpha(1-r^2),\\
				b_2&=(1-r^2)^2-4r^2\cos^2\alpha+r^4\sin^22\alpha,
				&
				b_1&={}-4r\cos\alpha(1-r^2),\\
				b_0&={}-(1-r^2)^2. 
			\end{aligned}
		\end{equation*}

		\textbf{Resultant computation:}
		The equations \(F_1=0\) and \(F_2=0\) in \eqref{p1f1andf2} admit a common solution in \(s\) if and only if their resultant with respect to \(s\) vanishes\cite{salmon1859lessons}. Since both polynomials are linear in \(s\), the resultant is given explicitly by
		\[
		\operatorname{Res}_s(F_1,F_2)
		=
		\begin{vmatrix}
			g_1 & g_0\\
			h_1 & h_0
		\end{vmatrix}
		=
		g_1h_0-g_0h_1=0.
		\]
		Substituting the corresponding coefficients \(g_0,g_1,h_0,\) and \(h_1\), we obtain
		\begin{align*}
			\operatorname{Res}_s(F_1,F_2)
			& =\,
			(-4r^3\cos\alpha\sin2\alpha)
			(-4r^3\,c\cos\alpha\sin2\alpha)
			\\
			&\quad {} -
			\Bigl(
			r^4\sin^22\alpha
			-2(1-r^2)(1+r^2\cos2\alpha)+(1-r^2)^2
			\\
			&\qquad +4r^2\cos^2\alpha-4r\cos\alpha(1-r^2)c
			-4r^2\,c^2\cos^2\alpha
			\Bigr)
			\\
			&\quad \times
			\Bigl(
			4r\cos\alpha\bigl((1-r^2)+2r\,c\cos\alpha\bigr)
			\Bigr)=0.
		\end{align*}
		After expanding, simplifying and collecting like powers of \(c\), the resultant is written as
		\begin{equation}\label{p1res1}
			\begin{split}
				\operatorname{Res}_s(F_1,F_2)
				& =a_3c^3+a_2c^2+a_1c+a_0=R_1(c)
			\end{split}
		\end{equation}
		where the coeffecients \(a_i\) are given by
		\begin{align*}
			a_3&= 32r^4\cos^4\alpha\\
			a_2&= 48r^3\cos^3\alpha(1-r^2)\\
			a_1&=16r^6\sin^22\alpha \cos^2\alpha
			+16r^2\cos^2\alpha(1-r^2)^2
			\\
			&\qquad
			-8r^2\cos^2\alpha
			\big(
			r^4\sin^22\alpha
			-2(1-r^2)(1+r^2\cos2\alpha)
			\\
			&\qquad
			+(1-r^2)^2
			+4r^2\cos^2\alpha
			\big)\\
			a_0&={}-4r\cos\alpha(1-r^2)
			\Big(
			r^4\sin^22\alpha
			-2(1-r^2)(1+r^2\cos2\alpha)
			\\
			&\qquad
			+(1-r^2)^2
			+4r^2\cos^2\alpha
			\Big)
		\end{align*}
		\noindent
		To eliminate the variable \(c\) from \eqref{p1res1} and \eqref{p1geqn}, we compute the resultant of the polynomials. Since \(R_1\) is a polynomial in \(c\) of order \(3\) and \(G\) is a polynomial in \(c\) of order \(4,\) the resultant is given by the determinant of the corresponding \emph{Sylvester} matrix  of order \(7\). The resultant vanishes precisely when those two polynomials have a common zero. Accordingly, the elimination condition is given by
		\begin{equation*} 
			\operatorname{Res}_c(R_1,G)=\left|
			\begin{array}{ccccccccccccc}
				a_3&a_2&a_1&a_0&0&0&0\\
				0&a_3&a_2&a_1&a_0&0&0\\
				0&0&a_3&a_2&a_1&a_0&0\\
				0&0&0&a_3&a_2&a_1&a_0\\
				b_4&b_3&b_2&b_1&b_0&0&0\\
				0&b_4&b_3&b_2&b_1&b_0&0\\
				0&0&b_4&b_3&b_2&b_1&b_0
			\end{array}
			\right|=0,
		\end{equation*}  which on substituting and evaluating the coefficients produces
		\begin{align*}
			\operatorname{Res}_c(R_1,G)&=\;131072 (-1 + r)^5 r^{18} (1 + r)^5 \cos^{14}\alpha\, \sin^2\alpha \bigl(-4 + 24 r^2 - 40 r^4 + 35 r^6 \\
			&\qquad{}+( 4-12r^2-6r^4+40r^6)\cos2\alpha +(- 12 r^2+24r^4+4r^6)\cos4\alpha \\
			&\qquad{}+
			(14 r^4-8r^6)\cos6\alpha  + r^6 \cos8\alpha \bigr)=0.
		\end{align*}
		Extracting the common non-vanishing factor
		\[
		131072\,(\cos^{14}\alpha)(\sin^2\alpha)\,(-1 + r)^5 r^{18} (1 + r)^5, \qquad \left(0<|\alpha| < \frac{\pi}{2}\right)
		\]
		we introduce the normalized polynomial
		\[
		T_{par}(r,\alpha)
		:=
		\frac{
			\operatorname{Res}_c(R_1,G)
		}{
			131072\,(\cos^{14}\alpha)(\sin^2\alpha)\,(-1 + r)^5 r^{18} (1 + r)^5
		}.
		\]
		Consequently, the resultant equation reduces to \eqref{p1redeq}. It follows that the admissible radius \(	R_{\mathcal S_{par}}(\mathcal{SP}_{\alpha})\) is precisely the smallest root of \eqref{p1redeq} lying in the interval \((0,1)\). Indeed, this value corresponds to the first occurrence of tangency between the boundary circle and the parabola defining \(\partial\Omega_{par}\). Eventually, \(	R_{\mathcal S_{par}}(\mathcal{SP}_{\alpha})\) represents the largest radius for which the associated disk is completely contained in \(\Omega_{par}\).
	\end{proof}

	\noindent For some specific values of $\alpha$ we have the following results.
	
	\begin{corollary}
		Let $ 	R_{\mathcal S_{par}}(\mathcal{SP}_{\alpha})$ be the radius of parabolic starlikeness for the class
		$\mathcal{SP}_\alpha$. Then  \begin{align*}
			R_{\mathcal S_{par}}(\mathcal{SP}_{\pi/3}) &= 0.455360692140855,
			& 	R_{\mathcal S_{par}}(\mathcal{SP}_{\pi/4}) &= 0.3826834323650898,\\
			R_{\mathcal S_{par}}(\mathcal{SP}_{\pi/6}) &= 0.3477134034204693,
			& 	R_{\mathcal S_{par}}(\mathcal{SP}_{\pi/8}) &= 0.33913699119761614.
		\end{align*}These radii are sharp.
	\end{corollary}
	
	\begin{proof}
		For \( \alpha=\pi/3 \), the resultant equation takes the form
		\[
		T_{par}\left(r,\frac{\pi}{3}\right)
		=
		-6 + 36 r^2 - 35 r^4 + \frac{9 r^6}{2}
		=0.
		\]
		Accordingly, the admissible radius \(	R_{\mathcal S_{par}}(\mathcal{SP}_{\pi/3})\) is obtained as the least root of \(T\left(r,\pi/3\right)\) in the interval \((0,1)\), namely
		\[
		R_{\mathcal S_{par}}(\mathcal{SP}_{\pi/3})=0.455360692140855.
		\]

		For \( \alpha=\pi/6 \), the resultant equation simplifies to
		\[
		T_{par}\left(r,\frac{\pi}{6}\right)
		=
		-2 + 24 r^2 - 69 r^4 + \frac{121}{2}r^6
		=0.
		\]
		Consequently, the admissible radius \(	R_{\mathcal S_{par}}(\mathcal{SP}_{\pi/6})\) is characterized as the smallest zero of \(T_{par}\left(r,\pi/6\right)\) contained in the interval \((0,1)\) which is given by
		\[
		R_{\mathcal S_{par}}(\mathcal{SP}_{\pi/6})=0.3477134034204693.
		\]

		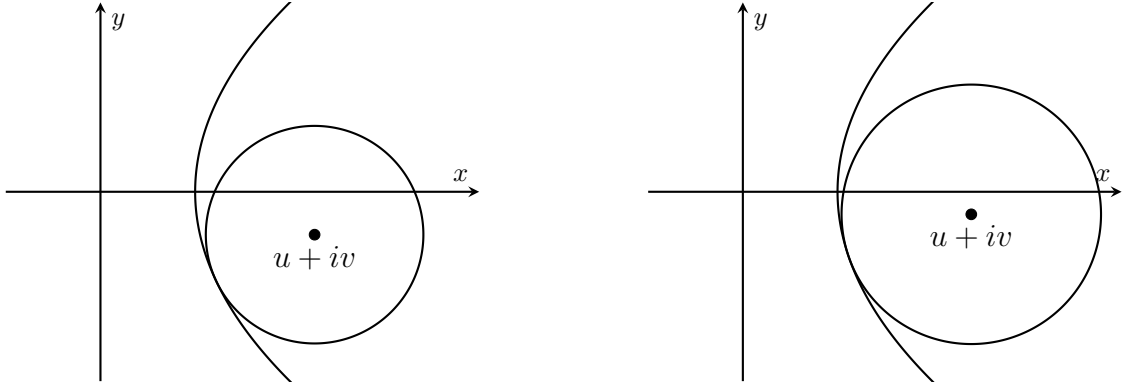
\begin{figure}[htbp]
			\centering
			
			\begin{minipage}{0.48\textwidth}
				\centering
				\begin{tikzpicture}
					\begin{axis}
						[
						width=\linewidth,
						axis line style={thick},
						axis lines=middle,
						xmin=-0.5, xmax=2,
						ymin=-1, ymax=1,
						axis equal image,
						samples=200,
						xtick=\empty,
						ytick=\empty,
						xlabel={$x$},
						ylabel={$y$},
						label style={font=\footnotesize},
						]
						
						% Parabola y^2 = 2x - 1
						\addplot[
						black,
						thick,
						domain=-1.2:1.2,
						samples=400
						]
						(
						{(x^2+1)/2},
						{x}
						);
						
						% Extremal circle
						\addplot[
						black,
						thick,
						domain=0:360,
						samples=400
						]
						(
						{1.1307981069175628 + 0.5744813251342448*cos(x)},
						{-0.22654896671504468 + 0.5744813251342448*sin(x)}
						);
						
						% Centre
						\addplot[
						only marks,
						mark=*,
						mark size=2pt
						]
						coordinates {(1.1307981069175628,-0.22654896671504468)};
						
						\node[below]
						at (axis cs:1.1307981069175628,-0.22654896671504468)
						{$u+iv$};
						
					\end{axis}
				\end{tikzpicture}
			\end{minipage}\hfill
			\begin{minipage}{0.48\textwidth}
				\centering
				\begin{tikzpicture}
					\begin{axis}
						[
						width=\linewidth,
						axis line style={thick},
						axis lines=middle,
						xmin=-0.5, xmax=2,
						ymin=-1, ymax=1,
						axis equal image,
						samples=200,
						xtick=\empty,
						ytick=\empty,
						xlabel={$x$},
						ylabel={$y$},
						label style={font=\footnotesize},
						]
						% Parabola y^2 = 2x - 1
						\addplot[
						black,
						thick,
						domain=-1.2:1.2,
						samples=400
						]
						(
						{(x^2+1)/2},
						{x}
						);
						
						% Extremal circle
						\addplot[
						black,
						thick,
						domain=0:360,
						samples=400
						]
						(
						{1.2062994740159003 + 0.6850875214190641*cos(x)},
						{-0.11910705685675822+ 0.6850875214190641*sin(x)}
						);
						
						% Centre
						\addplot[
						only marks,
						mark=*,
						mark size=2pt
						]
						coordinates {(1.2062994740159003,-0.11910705685675822)};
						
						\node[below]
						at (axis cs:1.2062994740159003,-0.11910705685675822)
						{$u+iv$};
						
					\end{axis}
				\end{tikzpicture}
			\end{minipage}
			
			\caption{The extremal disk tangent to the parabola \(y^2=2x-1\) for \(\alpha=\pi/3, \pi/6\).}
			\label{fig:parabolic-pi6}
			
		\end{figure}
		For \( \alpha=\pi/4 \), the resultant equation reduces to
		\[
		T_{par}\left(r,\frac{\pi}{4}\right)
		=
		4\left(-1 + 9r^2 - 16r^4 + 8r^6\right)
		=0.
		\]
		Hence, the admissible radius \(	R_{\mathcal S_{par}}(\mathcal{SP}_{\pi/4})\) is determined by the smallest root of \(T_{par}\left(r,\pi/4\right)\) belonging to the interval \((0,1)\). A numerical computation yields
		\[
		R_{\mathcal S_{par}}(\mathcal{SP}_{\pi/4})= 0.3826834323650898.
		\]
		
		For \( \alpha=\pi/8 \), the resultant equation assumes the form
		\[
		T_{par}\left(r,\frac{\pi}{8}\right)
		=
		-4 + 2\sqrt{2}
		- 6(-4+\sqrt{2})r^2
		- 10(4+\sqrt{2})r^4
		+ (34+24\sqrt{2})r^6
		=0.
		\]
		Accordingly, the admissible radius \(	R_{\mathcal S_{par}}(\mathcal{SP}_{\pi/8})\) is identified as the smallest root of \(T_{par}\left(r,\pi/8\right)\) contained in the interval \((0,1)\). Numerically, one obtains
		\[
		R_{\mathcal S_{par}}(\mathcal{SP}_{\pi/8})=0.33913699119761614.\qedhere
		\]
		
		\begin{figure}[htbp]
			\centering
			\begin{minipage}{0.48\textwidth}
				\centering
				\begin{tikzpicture}
					\begin{axis}
						[
						width=\linewidth,
						axis line style={thick},
						axis lines=middle,
						xmin=-0.5, xmax=2,
						ymin=-1, ymax=1,
						axis equal image,
						samples=200,
						xtick=\empty,
						ytick=\empty,
						xlabel={$x$},
						ylabel={$y$},
						label style={font=\footnotesize},
						]
						
						% Parabola y^2 = 2x - 1
						\addplot[black, thick, domain=-1.2:1.2, samples=200]
						({(x^2+1)/2}, {x});
						
						% Extremal circle
						\addplot[black, thick, domain=0:360, samples=300]
						({1.17157287525381 + 0.6340506711244288*cos(x)},
						{-0.1715728752538099 + 0.6340506711244288*sin(x)});
						
						% Centre
						\addplot[only marks, mark=*, mark size=1.5pt]
						coordinates {(1.17157287525381,-0.1715728752538099)};
						
						\node[below, font=\footnotesize]
						at (axis cs:1.17157287525381,-0.1715728752538099) {$u+iv$};
					\end{axis}
				\end{tikzpicture}
			\end{minipage}\hfill
			\begin{minipage}{0.48\textwidth}
				\centering
				\begin{tikzpicture}
					\begin{axis}
						[
						width=\linewidth,
						axis line style={thick},
						axis lines=middle,
						xmin=-0.5, xmax=2,
						ymin=-1, ymax=1,
						axis equal image,
						samples=200,
						xtick=\empty,
						ytick=\empty,
						xlabel={$x$},
						ylabel={$y$},
						label style={font=\footnotesize},
						]
						
						% Parabola y^2 = 2x - 1
						\addplot[black, thick, domain=-1.2:1.2, samples=200]
						({(x^2+1)/2}, {x});
						
						% Extremal circle
						\addplot[black, thick, domain=0:360, samples=300]
						({1.2218577289555161 + 0.7080828149946824*cos(x)},
						{-0.09189648025066896 + 0.7080828149946824*sin(x)});
						
						% Centre
						\addplot[only marks, mark=*, mark size=1.5pt]
						coordinates {(1.2218577289555161,-0.09189648025066896)};
						
						\node[below, font=\footnotesize]
						at (axis cs:1.2218577289555161,-0.09189648025066896) {$u+iv$};
					\end{axis}
				\end{tikzpicture}
			\end{minipage}
			\caption{The extremal disk tangent to the parabola \(y^2=2x-1\) for \(\alpha=\pi/4, \pi/8\).}
			\label{fig:parabolic-pi8}
		\end{figure}
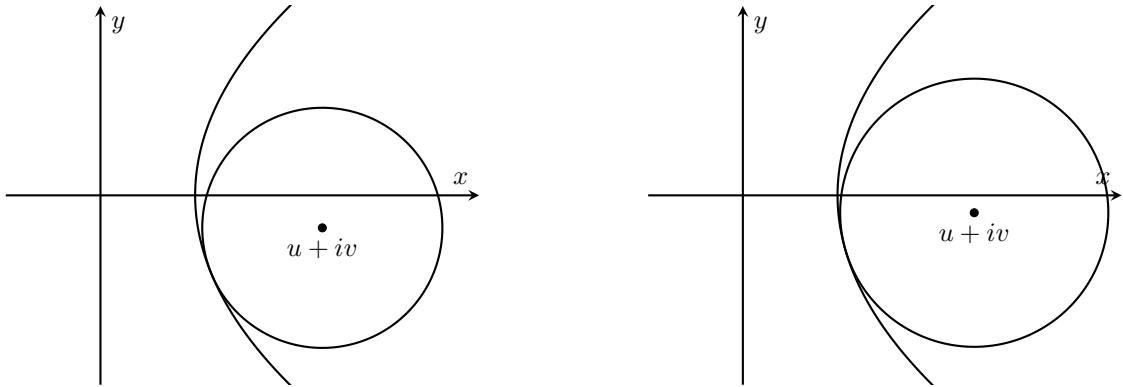
	\end{proof}
	
	The case $\alpha=0 $ reduces the class \(\mathcal{SP}_\alpha\) to the class of starlike functions for which we already have results \cite{MR1415180}.  However, the tangency approach needs a separate discussion. We shall prove the following.
	\begin{corollary}  The radius of parabolic starlikeness for the class of starlike functions is
		$ 	R_{\mathcal S_{par}}(\mathcal{SP}_{0}) = 1/3$.
		This radius is sharp.
	\end{corollary}
	\begin{proof}
		For $\alpha =0$, the elimination of $s$ and $c$ happens in the first step and
		\[u(r)=\frac{1+r^2}{1-r^2},
		\qquad
		v(r)=0,
		\qquad
		R(r)=\frac{2r}{1-r^2}.\]
		From equation \eqref{p1f1andf2}, we have
		\[F_1= R^2s^2-2u-2Rc+1=0,  \qquad
		F_2=2Rs(Rc+1)=0.\]Both equations must vanish simultaneously for some value \(t=t_0\). Since
		$F_2=s(Rc+1)$,
		we must have either \(s=0\) or \(Rc+1=0\).
		First suppose that \(Rc+1=0\). Then
		$c=-1/R$,
		which is possible only when \(R\ge1\), equivalently \(r>\sqrt2-1\). Substituting
		\(c=-1/R\) into \(F_1\) and simplifying yields $
		R^2-2u+2=0$.
		Upon replacing \(R\) and \(u\) by their expressions in terms of \(r\), we obtain
		\(r=0\), a contradiction. Hence \(Rc+1\neq0\).
		Therefore \(s=0\). Since \(s^2+c^2=1\), it follows that \(c=\pm1\). Substituting
		\(c=\pm1\) into \(F_1\) gives
		$
		R=\pm\left(u-1/2\right)$.
		Since \(R>0\), we must have \(c=-1\) and
		$		R=u-1/2$.		Substituting the expressions for \(R\) and \(u\) and solving for \(r\), we obtain
		$
		R_{\mathcal S_{par}}(\mathcal{SP}_{0})=1/3$.
	\end{proof}
	%\begin{figure}[htbp]
	%			\centering
	%			
	%			\begin{tikzpicture}
		%				\begin{axis}
			%					[
			%					axis lines=middle,
			%					xmin=-0.5, xmax=2,
			%					ymin=-1, ymax=1,
			%					axis equal image,
			%					samples=200,
			%					xlabel={$x$},
			%					ylabel={$y$},
			%					]
			%					
			%					% Parabola y^2 = 2x - 1
			%					\addplot[
			%					black,
			%					thick,
			%					domain=-1.2:1.2,
			%					samples=400
			%					]
			%					(
			%					{(x^2+1)/2},
			%					{x}
			%					);
			%					
			%					% Extremal circle
			%					\addplot[
			%					red,
			%					thick,
			%					domain=0:360,
			%					samples=500
			%					]
			%					(
			%					{1.25 + 0.75*cos(x)},
			%					{0 + 0.75*sin(x)}
			%					);
			%					
			%					% Centre
			%					\addplot[
			%					only marks,
			%					mark=*,
			%					mark size=2pt
			%					]
			%					coordinates {(1.25,0)};
			%					
			%					\node[above right]
			%					at (axis cs:1.25,0)
			%					{$u+iv$};
			%					
			%				\end{axis}
		%			\end{tikzpicture}
	%			
	%			\caption{The extremal disk tangent to the parabola \(y^2=2x-1\) for \(\alpha=0\).}
	%			\label{fig:parabolic-alpha0}
	%			
	%		\end{figure}

	\section{Radius of lemniscate starlikeness} In this section, we determine the  radius of lemniscate starlikeness for the class
	$\mathcal{SP}_\alpha$.  %For two polynomials
	%\[
	%f(x)=a_mx^m+a_{m-1}x^{m-1}+\cdots+a_0
	%\quad\text{and}\quad
	%g(x)=b_nx^n+b_{n-1}x^{n-1}+\cdots+b_0,
	%\]
	%their \emph{Sylvester} matrix, denoted by \(S(f,g)\), is the \((m+n)\times(m+n)\) matrix obtained by placing the coefficients of \(f\) in \(n\) successive rows and those of \(g\) in \(m\) successive rows, each row shifted one place to the right from the preceding one. The resultant of \(f\) and \(g\) is defined by
	%$
	%\operatorname{Res}(f,g)=\det S(f,g).
	%$
	%For example, if
	%\[
	%R_1(c)=a_3c^3+a_2c^2+a_1c+a_0
	%\quad\text{and}\quad
	%G(c)=b_4c^4+b_3c^3+b_2c^2+b_1c+b_0,
	%\]
	%then the resultant in \eqref{p1Silver1} is the determinant of the Sylvester matrix $S(R_1,G)$.
	We consider the following two polynomialls
	\begin{equation} \label{p1lempoly}
	R_1(c)=a_5c^5+\dotsb+ a_1c+a_0
	\quad\text{and}\quad
	G(c)=b_8c^8+\dotsb+b_1c+b_0,
	\end{equation}
	where the coefficients \(a_i \text{ and } b_i\) are given by
	\begin{align*}
	a_5& =128  r^8 \cos^2\alpha (1 + \cos2\alpha)\bigl(-6 + 2 r^2 +(r^2- 8) \cos2\alpha - (2+2r^2) \cos4\alpha \\
	&\quad {} - r^2 \cos6\alpha \bigr)\\
	a_4& ={} -128  r^7 \cos^2\alpha(1 + \cos2\alpha) \bigl(-6 r + 2 r^3 + (12+8r^2) \cos\alpha  +(r^3- 8 r) \cos2\alpha \\
	&\quad{} + (4+6r^2) \cos3\alpha - 2 (r+r^3) \cos4\alpha  + 2 r^2 \cos5\alpha - r^3 \cos6\alpha \bigr)\\
	a_3& =-32  r^4 \cos^2\alpha (-1 + 12 r^2 - 68 r^4 + 60 r^6 - 7 r^8 -16r^3\cos\alpha(3+2r^2) \\
	&\quad{}+( 20 r^2- 148 r^4+ 90 r^6- 3 r^8 - 48 r^3 \cos\alpha - 32 r^5 \cos\alpha + 8 r^2 \cos2\alpha \\
	&\quad{}- 80 r^4 \cos2\alpha  + 22 r^6 \cos2\alpha-1)\cos2\alpha  +(- 16 r^3 - 24 r^5- 16 r^3 \cos2\alpha \\
	&\quad{}- 24 r^5 \cos2\alpha) \cos3\alpha  +(- 18 r^4 - 40 r^6+ 14 r^8- 18 r^4 \cos2\alpha- 56 r^6 \cos2\alpha  \\
	&\quad{}+ 2 r^8 \cos2\alpha- 8 r^6 \cos 4\alpha- 8 r^8 \cos 4\alpha)\cos4\alpha  - 8 r^5\cos5\alpha - 8 r^5\cos2\alpha \cos5\alpha  \\
	&\quad{}+(- 14 r^6+ 4 r^8- 14 r^6 \cos2\alpha- 4 r^8 \cos4\alpha)\cos6\alpha -8(4 r^4+1)\sin^22\alpha \\
	&\quad{}+(- 16 r^4- 32 r^6- 4 r^8) \sin2\alpha\sin4\alpha  - 8 r^6\sin^24\alpha(1+r^2) \\
	&\quad{}-8r^6\sin2\alpha\sin6\alpha- 4 r^8 \sin4\alpha\sin6\alpha)\\
	a_2& =32 r^4 \cos^2\alpha \bigl(-1 + 12 r^2 - 20 r^4 + 20 r^6 + r^8 + (192 r^3+ 56 r^5- 40 r^7) \cos\alpha \\
	&\quad{} +(-1+ 20 r^2- 36 r^4+ 34 r^6+ r^8+ 208 r^3 \cos\alpha  + 120 r^5 \cos\alpha - 4 r^7\cos\alpha\\
	&\quad{} + 8 r^2 \cos2\alpha- 16 r^4 \cos2\alpha+ 14 r^6 \cos2\alpha)   \cos2\alpha+ (48 r^3 + 80 r^5 -
	32 r^7 \\
	&\quad{}+ 48 r^3 \cos2\alpha+ 104 r^5\cos2\alpha- 4 r^7 \cos2\alpha)\cos3\alpha  +( - 2 r^4 - 8 r^6- 2 r^8\\
	&\quad{}+ 16 r^3\cos\alpha + 72 r^5 \cos\alpha + 40 r^7 \cos\alpha- 2 r^4 \cos2\alpha- 8 r^6 \cos2\alpha - 2 r^8 \cos2\alpha\\
	&\quad{}+ 24 r^5 \cos3\alpha+ 32 r^7\cos3\alpha)\cos4\alpha  + (24 r^5- 8 r^7 + 24 r^5\cos2\alpha\\
	&\quad{}+ 8 r^7 \cos4\alpha)\cos5\alpha+(- 6 r^6 + 8 r^5 \cos\alpha+ 4 r^7\cos\alpha - 6 r^6 \cos2\alpha\\
	&\quad{}+ 4 r^7 \cos3\alpha ) \cos6\alpha + (32 r^3- 12 r^5+ 28 r^7+ 4 r^9)\sin\alpha \sin2\alpha + (16 r^3+ 32 r^5\\
	&\quad{}+ 20 r^7 ) \sin2\alpha \sin3\alpha  + (8 r^3 + 8 r^5+ 16 r^9 )\sin\alpha \sin4\alpha  + (16 r^5+ 16 r^7\\
	&\quad{}+ 8 r^9) \sin3\alpha \sin4\alpha  + (16 r^5 \sin2\alpha + 8 r^7 \sin4\alpha)\sin5\alpha + (4 r^5 - 4 r^7 \\
	&\quad{}+ 4 r^9)\sin\alpha \sin6\alpha + 4 r^7\sin3\alpha \sin6\alpha\bigr)\\
	a_1& = -32 r^4 \cos^2\alpha \bigl(3 - 37 r^2 - 32 r^6 + 9 r^8 + r^{10} + (48 r^3+ 32 r^5) \cos\alpha   \\
	&\quad{}+(- 96 r^2- 112 r^4 - 32 r^6) \cos^2\alpha    +( 3- 60 r^2- 52 r^4  - 50 r^6+ 7 r^8+ 48 r^3 \cos\alpha\\
	&\quad{}+ 32 r^5 \cos\alpha)\cos2\alpha +( - 24 r^2- 48 r^4- 30 r^6 ) \cos^22\alpha  + (16 r^3+ 24 r^5 \\
	&\quad{}- 32 r^2 \cos\alpha - 112 r^4 \cos\alpha  - 56 r^6 \cos\alpha+ 16 r^3 \cos2\alpha  + 24 r^5 \cos2\alpha )\cos3\alpha     \\
	&\quad{}+(- 16 r^4 - 24 r^6) \cos^23\alpha + (r^2- 30 r^4+ 2 r^6- 6 r^8 - 3 r^{10} - 26 r^4 \cos2\alpha\\
	&\quad{}- 16 r^6 \cos2\alpha- 4 r^8\cos2\alpha) \cos4\alpha  +(- 6 r^6+ 2 r^{10}) \cos^24\alpha  + 8 r^5 \cos5\alpha  \\
	&\quad{}+(- 16 r^4\cos\alpha- 8 r^6\cos\alpha+ 8 r^5 \cos2\alpha- 8 r^6  \cos3\alpha)  \cos5\alpha  +( 6 r^6  - 2 r^8\\
	&\quad{}+ 6 r^6 \cos2\alpha+ 2 r^8 \cos4\alpha)\cos6\alpha +( - 8 r^2+ 8 r^4- 16 r^6+ 8 r^8- 8 r^{10}) \sin^2\alpha \\
	&\quad{}+ (16 r^4 + 4 r^6)\sin^22\alpha  +(- 8 r^2 - 8 r^4 - 8 r^6- 8 r^8-8 r^{10})\sin\alpha\sin3\alpha\\
	&\quad{}+( - 8 r^4- 8 r^6- 8 r^8 ) \sin^23\alpha+ (8 r^4+ 16 r^6+ 2 r^8) \sin2\alpha\sin4\alpha   \\
	&\quad{}+ (4 r^6 + 4 r^8)\sin^24\alpha  +(- 8 r^4+ 8 r^6 - 8 r^8 )\sin\alpha \sin5\alpha   \\
	&\quad{}- 8 r^6 \sin3\alpha \sin5\alpha+ 4 r^6 \sin2\alpha \sin6\alpha + 2 r^8 \sin4\alpha \sin6\alpha\bigr)\\
	a_0&= -32 r^{3}\cos^{2}\alpha \bigl(-r+12r^{3}+4r^{5}+12r^{7}+r^{9}+(2-23r^{2}-20r^{4}-28r^{6}-14r^{8} \\
	&\quad{}-r^{10})\cos\alpha+(-r+20r^{3}+20r^{5}+22r^{7}+r^{9}-16r^{2}\cos\alpha-40r^{4}\cos\alpha \\
	&\quad{}-36r^{6}\cos\alpha-10r^{8}\cos\alpha)\cos2\alpha+(8r^{3}+16r^{5}+10r^{7})\cos^{2}2\alpha+(r^{2}-12r^{4}\\
	&\quad{}-4r^{6}-12r^{8}-r^{10}-8r^{4}\cos2\alpha-16r^{6}\cos2\alpha-10r^{8}\cos2\alpha)\cos3\alpha \\
	&\quad{}+(6r^{5}-2r^{9}-12r^{4}\cos\alpha-6r^{6}\cos\alpha+4r^{8}\cos\alpha+2r^{10}\cos\alpha+6r^{5}\cos2\alpha\\
	&\quad{}-2r^{9}\cos2\alpha-6r^{6}\cos3\alpha+2r^{10}\cos3\alpha)\cos4\alpha +(-2r^{7}+4r^{6}\cos\alpha+2r^{8}\cos\alpha\\
	&
	\quad{}-2r^{7}\cos2\alpha+2r^{8}\cos3\alpha)\cos6\alpha +(8r^{4}+8r^{8})\sin\alpha\sin2\alpha \\
	&\quad{}
	+(8r^{4}+8r^{6}+8r^{8})\sin2\alpha\sin3\alpha+(4r^{6}+4r^{10})\sin\alpha\sin4\alpha\\
	&
	\quad{}+(4r^{6}+4r^{8}+4r^{10})\sin3\alpha\sin4\alpha +8r^{6}\sin2\alpha\sin5\alpha +4r^{8}\sin4\alpha\sin5\alpha\bigr).\\
	b_8& =1024  r^8 \cos^4\alpha (1 + \cos2\alpha)^2\\
	b_7& =-2048  r^8 \cos^4\alpha (1 +\cos2\alpha)^2\\
	b_6& =-2048  r^8 \cos^4\alpha (1 + \cos2\alpha) (3 - r^2 + 3 \cos2\alpha + r^2 \cos4\alpha)\\
	b_5& =2048 r^{7}\cos^{4}\alpha (1+\cos2\alpha)
	(5r-r^{3}-2\cos\alpha-r^{2}\cos\alpha+5r\cos2\alpha\\
	&\quad{}-r^{2}\cos3\alpha+r^{3}\cos4\alpha)\\
	b_4& =1024 r^{7}\cos^{4}\alpha (12r-8r^{3}+r^{5}+(4+2r^{2})\cos\alpha +(24r-8r^{3}+4\cos\alpha\\
	&
	\quad{}+2r^{2}\cos\alpha+12r\cos2\alpha)\cos2\alpha +(2r^{2}+2r^{2}\cos2\alpha)\cos3\alpha +(8r^{3}-2r^{5}\\
	&
	\quad{}+8r^{3}\cos2\alpha+r^{5}\cos4\alpha)\cos4\alpha+4r\sin^{2}2\alpha +(4r^{3}\sin2\alpha+r^{5}\sin4\alpha)\sin4\alpha\bigr)\\
	b_3& =2048\,r^{7}\cos^{4}\alpha \bigl(-7r+2r^{3}+(8+2r^{2}-r^{4})\cos\alpha+(-14r+2r^{3}+8\cos\alpha\\
	&
	\quad{}+4r^{2}\cos\alpha-7r\cos2\alpha)\cos2\alpha +(4r^{2}-r^{4}+4r^{2}\cos2\alpha)\cos3\alpha \\
	&\quad{}+(-2r^{3}+2r^{2}\cos\alpha+r^{4}\cos\alpha-2r^{3}\cos2\alpha+r^{4}\cos3\alpha)\cos4\alpha\\
	&\quad{}+(2-2r^{2}+2r^{4})\sin\alpha\sin2\alpha+2r^{2}\sin2\alpha\sin3\alpha +(r^{2}-r^{4}+r^{6})\sin\alpha\sin4\alpha\\
	&
	\quad{}+r^{4}\sin3\alpha\sin4\alpha\bigr)\\
	b_2& =1024\,r^{6}\cos^{4}\alpha (-6r^{2}+6r^{4}+-r^{6}+(-8r-4r^{3})\cos\alpha+(4+4r^{2}+r^{4})\cos^{2}\alpha\\
	&
	\quad{}+(-12r^{2}+6r^{4}-8r\cos\alpha-4r^{3}\cos\alpha-6r^{2}\cos2\alpha)\cos2\alpha  \\
	&
	\quad{}+(-4r^{3}+4r^{2}\cos\alpha+2r^{4}\cos\alpha-4r^{3}\cos2\alpha+r^{4}\cos3\alpha)\cos3\alpha\\
	&
	\quad{}+(-6r^{4}+2r^{6}-6r^{4}\cos2\alpha-r^{6}\cos4\alpha)\cos4\alpha+(1-2r^{2}+3r^{4}-2r^{6} \\
	&
	\quad{}+r^{8})\sin^{2}\alpha-4r^{2}\sin^{2}2\alpha +(2r^{2}-2r^{4}+2r^{6})\sin\alpha\sin3\alpha\\
	&
	\quad{}+r^{4}\sin^{2}3\alpha-4r^{4}\sin2\alpha\sin4\alpha-r^{6}\sin^{2}4\alpha\bigr)\\
	b_1& =-1024\,r^{7}\cos^{4}\alpha (-6r+2r^{3}+(12+2r^{2}-2r^{4})\cos\alpha+(-12r+2r^{3}+12\cos\alpha\\
	&
	\quad{}+6r^{2}\cos\alpha-6r\cos2\alpha)\cos2\alpha +(6r^{2}-2r^{4}+6r^{2}\cos2\alpha)\cos3\alpha \\
	&
	\quad{}+(-2r^{3}+4r^{2}\cos\alpha+2r^{4}\cos\alpha-2r^{3}\cos2\alpha+2r^{4}\cos3\alpha)\cos4\alpha
	+(2-2r^{2}\\
	&
	\quad{}+2r^{4})\sin\alpha\sin2\alpha
	+2r^{2}\sin2\alpha\sin3\alpha +(r^{2}-r^{4}+r^{6})\sin\alpha\sin4\alpha
	\\
	&
	\quad{}+r^{4}\sin3\alpha\sin4\alpha\bigr)\\
	b_0&=  256\,r^{6}\cos^{4}\alpha\bigl(2r
	+(-4-2r^{2})\cos\alpha
	+2r\cos2\alpha
	-2r^{2}\cos3\alpha \\
	&
	\quad{}+2r\sin2\alpha
	+r^{3}\sin4\alpha\bigr)\bigl(	-2r
	+(4+2r^{2}\cos\alpha
	\cos\alpha \\
	&
	\quad{}-2r\cos2\alpha
	+2r^{2}\cos3\alpha
	+2r\sin2\alpha
	+r^{3}\sin4\alpha\bigr).
	\end{align*}
	The \textit{Sylvester } matrix for \(R_1 \text{ and } G\) is denoted by \(S(R_1,G).\)
	\begin{theorem}Let  $0<|\alpha|<\pi/2$.  The radius of lemniscate starlikeness for the class
	$\mathcal{SP}_\alpha$  is the smallest admissible root in $(0,1)$ of  \(T_L(r,\alpha)=\det S(R_1 ,G )=0,\)  where
	$R_1, G$ are given by \eqref{p1lempoly}.		
	\end{theorem}  %The radius of lemniscate starlikeness for the class
	%		$\mathcal{SP}_\alpha$  is the smallest root in $(0,1)$ of
	%		\begin{align}\label{redeq2}
	%			T_L(r,\alpha)=\left|
	%			\begin{array}{ccccccccccccc}
		%				a_5&a_4&a_3&a_2&a_1&a_0&0&0&0&0&0&0&0\\
		%				0&a_5&a_4&a_3&a_2&a_1&a_0&0&0&0&0&0&0\\
		%				0&0&a_5&a_4&a_3&a_2&a_1&a_0&0&0&0&0&0\\
		%				0&0&0&a_5&a_4&a_3&a_2&a_1&a_0&0&0&0&0\\
		%				0&0&0&0&a_5&a_4&a_3&a_2&a_1&a_0&0&0&0\\
		%				0&0&0&0&0&a_5&a_4&a_3&a_2&a_1&a_0&0&0\\
		%				0&0&0&0&0&0&a_5&a_4&a_3&a_2&a_1&a_0&0\\
		%				0&0&0&0&0&0&0&a_5&a_4&a_3&a_2&a_1&a_0\\
		%				b_8&b_7&b_6&b_5&b_4&b_3&b_2&b_1&b_0&0&0&0&0\\
		%				0&b_8&b_7&b_6&b_5&b_4&b_3&b_2&b_1&b_0&0&0&0\\
		%				0&0&b_8&b_7&b_6&b_5&b_4&b_3&b_2&b_1&b_0&0&0\\
		%				0&0&0&b_8&b_7&b_6&b_5&b_4&b_3&b_2&b_1&b_0&0\\
		%				0&0&0&0&b_8&b_7&b_6&b_5&b_4&b_3&b_2&b_1&b_0
		%			\end{array}
	%			\right|=0
	%		\end{align}

	\begin{proof}
	Let \( \Omega_{L}=\{ w\in\mathbb{C}: |w^2-1|<1 \}\). Then the boundary of $\Omega_{L}$ is given by \[ |w^2 -1| =1. \]
	Let \(w=x+iy,\) then the boundary of  $\Omega_{L}$ will become
	\( |x^2-y^2+2xyi-1|=1\), which produces
	\[ \Phi (x,y):= \left( x^2-y^2-1\right) ^2+4x^2y^2-1=0.
	\]
	We define the function $F: (0,1) \times [0,2\pi] \to \mathbb{R}$ along
	the curve $\gamma(t) $ in \eqref{p1gammaeqn} as
	\begin{equation} \label{p1fdefnlem}
		F(r,t) := \Phi(\gamma(t)) = \left( x^2(r,t)-y^2(r,t)-1\right) ^2+4x^2(r,t)y^2(r,t)-1
	\end{equation}
	The sign of \(F(r,t)\) determines the position of the point \(\gamma(t)\) relative to the region \(\Omega_{L}\). More precisely, the disk determined by \(\gamma(t)\) is contained in \(\Omega_L\) when
	\[
	F(r,t)\le 0
	\qquad \text{for all } t.
	\]
	The extremal situation occurs when equality is attained for at least one value of \(t\).
	We want the disk \eqref{p1reqdisk} to lie inside $\Omega_{L}$, therefore we have to prove the tangency condition for the boundary curves $\Phi $ and $\partial\Omega_{L}.$

	\textbf{Tangency condition: } At the extremal radius, the boundary circle is tangent to the lemniscate defining \(\partial\Omega_{L}\). As a result, there exists \(t=t_0\) such that
	\begin{align} \label{p1cond1lem}
		F(r,t)&=0,\\ \label{p1cond2lem}
		\frac{\partial F(r,t)}{\partial t}&=0.
	\end{align}
	Evaluating \eqref{p1cond1lem} and \eqref{p1cond2lem} using \eqref{p1fdefnlem}, \eqref{p1gammaeqn} and the identity \(s^2=1-c^2\),  we obtain
	\begin{align} \label{p1eqn1lem}
		F(r,t)& =2 R^2 - 4 c^2 R^2 + R^4 - 4 c R u + 4 c R^3 u - 2 u^2 + 2 R^2 u^2  \\  \notag
		&\quad {}+
		4 c^2 R^2 u^2 + 4 c R u^3 + u^4 + 2 v^2 + 6 R^2 v^2 -
		4 c^2 R^2 v^2 + 4 c R u v^2  \\ \notag
		&\quad {}+ 2 u^2 v^2 + v^4+ 4 R s v (1 + R^2 + 2 c R u + u^2 + v^2)=0 \\
		\label{p1eqn2lem}
		F_t(r,t)& =4R\bigl((v (c + c R^2 - 2 R u + 4 c^2 R u + c u^2 + c v^2)\\ \notag
		&\quad {} +
		s (2 c R + R^3 - c R^3 - c^2 R^3 + c^3 R^3 + u - R^2 u -
		2 c R u^2 - u^3 \\
		&\quad{}+ 2 c R v^2 - u v^2)\bigr)=0 \notag
	\end{align}
	From \(F_t(r,t)\), on computing $s$ we have
	\[s=\frac{-v (c + c R^2 - 2 R u + 4 c^2 R u + c u^2 + c v^2)}{(2 c R + R^3 - c R^3 - c^2 R^3 + c^3 R^3 + u - R^2 u -
		2 c R u^2 - u^3 + 2 c R v^2 - u v^2)}.\]
	If the denominator vanishes, then it follows from \eqref{p1eqn2lem} that
	\[
	F_2=4R v (c + c R^2 - 2 R u + 4 c^2 R u + c u^2 + c v^2)=0.
	\]
	Solving this equation together with \(F_1=0\) yields solutions satisfying \(r>	R_{\mathcal S_{L}}(\mathcal{SP}_{\alpha})\). Since \(	R_{\mathcal S_{L}}(\mathcal{SP}_{\alpha})\) is the smallest radius at which tangency occurs, every such solution corresponds to a tangency that takes place after the disk has already ceased to be contained in \(\Omega_{L}\). Consequently, the case
	\(
	(2 c R + R^3 - c R^3 - c^2 R^3 + c^3 R^3 + u - R^2 u -
	2 c R u^2 - u^3 + 2 c R v^2 - u v^2)=0
	\)
	may be excluded from further consideration. Substituting the resulting expression for \(s\) into the identity \(s^2 = 1 - c^2\), we arrive at
	\begin{align} \label{p1geqnlem}
		G(c):& =\, -32 c^7 R^8 + 16 c^8 R^8 - 32 c^6 R^6 (-2 + R^2 + 2 u^2 - 2 v^2)\\ \notag
		& \quad{}- 16 R^2 (R^3 + u - R^2 u - u^3 - 2 R u v - u v^2) (R^3 + u - R^2 u \\ \notag
		&\quad{}-
		u^3 + 2 R u v - u v^2) +
		32 c^5 R^5 (-2 R + 3 R^3 + u - R^2 u + 2 R u^2 \\ \notag
		&\quad{}- u^3 - 2 R v^2 -
		u v^2) +
		32 c R^3 (-2 R^3 + R^5 - 2 u + 3 R^2 u - R^4 u \\ \notag
		&\quad{}+ 2 R^3 u^2 + 4 u^3 -
		3 R^2 u^3 - 2 u^5 - 2 R^3 v^2 - 2 u v^2 - R^2 u v^2 -
		2 u^3 v^2) \\ \notag
		&\quad{}+
		32 c^4 R^4 (2 - 4 R^2 - R u + R^3 u - 4 u^2 + 4 R^2 u^2 + R u^3 +
		2 u^4 + 4 v^2 \\ \notag
		&\quad{}- 4 R^2 v^2 + R u v^2 + 4 u^2 v^2 + 2 v^4) -
		32 c^3 R^3 (-4 R^3 + 3 R^5 - 2 u + 4 R^2 u \\ \notag
		&\quad{}- 2 R^4 u + 4 R^3 u^2 +
		4 u^3 - 4 R^2 u^3 - 2 u^5 - 4 R^3 v^2 - 4 u v^2 - 4 R^2 u v^2 \\ \notag
		&\quad{}-
		4 u^3 v^2 - 2 u v^4) +
		16 c^2 R^2 (-4 R^2 + 4 R^4 + 2 R^6 + 4 R^3 u - 4 R^5 u + u^2 \\ \notag
		&\quad{}+
		6 R^2 u^2 - 3 R^4 u^2 - 4 R^3 u^3 - 2 u^4 - 2 R^2 u^4 + u^6 +
		v^2 - 6 R^2 v^2 \\ \notag
		&\quad{}+ 5 R^4 v^2 - 4 R^3 u v^2 - 4 R^2 u^2 v^2 +
		3 u^4 v^2 + 2 v^4 - 2 R^2 v^4 + 3 u^2 v^4 + v^6)
	\end{align}
	\textbf{Resultant computation:}  For eliminating $s$ in equations \eqref{p1eqn1lem} and \eqref{p1eqn2lem}, we compute the resultant considering them as a linear equation in $s$,   we write
	\[
	F=g_0+g_1s,
	\qquad
	F_t=h_0+h_1s.
	\]
	where
	\begin{align*} g_0& =2 R^2 - 4 c^2 R^2 + R^4 - 4 c R u + 4 c R^3 u - 2 u^2 + 2 R^2 u^2 +
		4 c^2 R^2 u^2 + 4 c R u^3 \\
		&\quad{}+ u^4 + 2 v^2 + 6 R^2 v^2 -
		4 c^2 R^2 v^2 + 4 c R u v^2 + 2 u^2 v^2 + v^4, 	\\
		g_1& =  4 R v (1 + R^2 + 2 c R u + u^2 + v^2)\\
		h_0& = v (c + c R^2 - 2 R u + 4 c^2 R u + c u^2 + c v^2)\\
		h_1& = (2 c R + R^3 - c R^3 - c^2 R^3 + c^3 R^3 + u - R^2 u -
		2 c R u^2 - u^3 + 2 c R v^2 - u v^2)
	\end{align*}	
	The equations \(F=0\) and \(F_t=0\) admit a common solution in \(s\) if and only if their resultant with respect to \(s\) vanishes. Then the resultant is given by
	\[
	\operatorname{Res}_s(F,F_t)
	=
	\begin{vmatrix}
		g_1 & g_0\\
		h_1 & h_0
	\end{vmatrix}
	=
	g_1h_0-g_0h_1=0.
	\]
	Substituting the corresponding coefficients \(g_0,g_1,h_0,\) and \(h_1\), we obtain
	\[\begin{aligned}
		Res_s(F,F_t)& =-16 c^5 R^6 (-1 + u^2 - v^2) -
		16 c^4 R^5 (R - u + R^2 u - R u^2 + u^3 \\
		&\quad{}+ R v^2 + u v^2) -
		4 c^3 R^4 (-8 + 6 R^2 + R^4 + 4 R u - 4 R^3 u + 14 u^2 \\
		&\quad{}- 2 R^2 u^2 -
		4 R u^3 - 7 u^4 - 14 v^2 + 10 R^2 v^2 - 4 R u v^2 - 14 u^2 v^2 \\
		&\quad{}-
		7 v^4) +
		4 c^2 R^3 (6 R^3 + R^5 + 12 u - 16 R^2 u + 4 R^4 u - 2 R u^2 \\
		&\quad{}-
		2 R^3 u^2 - 24 u^3 + 16 R^2 u^3 + R u^4 + 12 u^5 + 2 R v^2 +
		10 R^3 v^2 \\
		&\quad{}+ 24 u v^2 + 16 R^2 u v^2 + 2 R u^2 v^2 + 24 u^3 v^2 +
		R v^4 + 12 u v^4) \\
		&\quad{}+
		4 c R^2 (-4 R^2 + R^6 + 4 R^3 u - 4 R^5 u + 8 u^2 - 10 R^2 u^2 +
		8 R^4 u^2 \\
		&\quad{}- 4 R^3 u^3 - 14 u^4 + 13 R^2 u^4 + 6 u^6 - 6 R^2 v^2 +
		8 R^4 v^2 - 4 R^3 u v^2 \\
		&\quad{}+ 4 u^2 v^2 + 10 R^2 u^2 v^2 +
		14 u^4 v^2 + 2 v^4 - 3 R^2 v^4 + 10 u^2 v^4 + 2 v^6) \\
		&\quad{}-
		4 R (2 R^5 + R^7 + 2 R^2 u - R^4 u - R^6 u - 2 R^3 u^2 + 2 R^5 u^2 -
		2 u^3 \\
		&\quad{}+ 2 R^2 u^3 - 3 R^4 u^3 + R^3 u^4 + 3 u^5 - 3 R^2 u^5 -
		u^7 + 2 R^3 v^2 + 6 R^5 v^2 \\
		&\quad{}+ 2 u v^2 + 10 R^2 u v^2 +
		R^4 u v^2 + 2 R^3 u^2 v^2 + 2 u^3 v^2 - 2 R^2 u^3 v^2 \\
		&\quad{}-
		3 u^5 v^2 + R^3 v^4 - u v^4 + R^2 u v^4 - 3 u^3 v^4 - u v^6)=0.
	\end{aligned}\]
	\(Res_s(F,F_t)\) is a polynomial in \(c\) of degree 5 and from equation \eqref{p1geqnlem} we can see that \(G(c)\) is a polynomial in \(c\) of degree 8.  Substituting the values of \(u,v,R\), canceling the denominator \((-1 + r)^4(1 + r)^4\) and simplifying we get the following polynomials in \(c\)
	\begin{align*}G(c)&=b_8c^8+b_7c^7+b_6c^6+b_5c^5+b_4c^4+b_3c^3+b_2c^2+b_1c+b_0\\
		Res_s(F,F_t)&=a_5c^5+a_4c^4+a_3c^3+a_2c^2+a_1c+a_0=R_1(c)\end{align*}
	where the \(a_i's \) and \(b_i's\) are given in \eqref{p1lempoly}. The two equations
	\(Res_s(F,F_t)=0\) and \(G(c)=0\) have a common solution in \(c\) if and only if their resultant \(T_L(r,\alpha)= Res_c(R_1(c),G(c))\) with respect to \(c\) vanishes.  The resultant with respect to \(c\) for these two polynomial will be the determinant of the corresponding \textit{Sylvester} matrix which is of order 13, whose entries are given by equation \eqref{p1lempoly}.  The least admissible zero of this polynomial, \(	R_{\mathcal S_{L}}(\mathcal{SP}_{\alpha})\) lying in \((0,1)\) is the radius of lemniscate starlikness for the class \(\mathcal{SP_\alpha}\).
	\end{proof}
	
	In the following results, the lemniscate starlikeness radius for the class  \(\mathcal{SP_\alpha}\) are given for some particular values of alpha.
	
	\begin{corollary}
	Let $ 	R_{\mathcal S_{L}}(\mathcal{SP}_{\alpha})$ be the radius of lemniscate starlikeness for the class
	$\mathcal{SP}_\alpha$. Then      $ 	R_{\mathcal S_{L}}(\mathcal{SP}_{\pi/3})=0.3029058506102136$,   $ 	R_{\mathcal S_{L}}(\mathcal{SP}_{\pi/6})=0.19459383342676778$ and \\ $ 	R_{\mathcal S_{L}}(\mathcal{SP}_{\pi/4})=0.2308777657345702$.
	These radii are sharp.
	\end{corollary}
	\begin{proof}
	For \(\alpha=\pi/3\), the final resultant equation
	\begin{align*}
		T_L\left(r,\frac{\pi}{3}\right)& =28991029248 \,r^{64}(16 + 24 r + 112 r^2 - 96 r^3 - 316 r^4 + 24 r^5 +
		536 r^6 \\
		&\quad{}+ 210 r^7 - 586 r^8 - 318 r^9 + 359 r^{10} + 174 r^{11} -
		116 r^{12} - 36 r^{13} \\
		&\quad{}+ 16 r^{14})^2 (1 - 64 r^2 - 16 r^3 + 1880 r^4 + 904 r^5 - 31960 r^6 - 20752 r^7  \\
		&\quad{}+
		334444 r^8+ 238648 r^9 - 2030008 r^{10} - 1128504 r^{11} +
		4102948 r^{12}  \\
		&\quad{}- 4795496 r^{13}+ 39065240 r^{14} + 105663248 r^{15} -
		409154834 r^{16}  \\
		&\quad{}- 735734144 r^{17}+ 1927302944 r^{18} +
		2865535784 r^{19} - 4755979828 r^{20} \\
		&\quad{}- 5531481392 r^{21} -
		1041874128 r^{22} - 6649989016 r^{23} \\
		&\quad{}+ 67151149010 r^{24} +
		94011349768 r^{25} - 349292098376 r^{26} \\
		&\quad{}- 399632973720 r^{27} +
		1165510085476 r^{28} + 1143299362928 r^{29} \\
		&\quad{}- 2946243593360 r^{30} -
		2505764783472 r^{31} + 5972623978205 r^{32} \\
		&\quad{}+ 4419128143992 r^{33} -
		9977066382200 r^{34} - 6428543371744 r^{35} \\
		&\quad{}+ 13940876971360 r^{36} +
		7825316192592 r^{37} - 16434720568112 r^{38} \\
		&\quad{}- 8040973315448 r^{39} +
		16421302072718 r^{40} + 7010633287712 r^{41} \\
		&\quad{}- 13929439011648 r^{42} -
		5197931424392 r^{43} + 10022506126364 r^{44} \\
		&\quad{}+ 3275987034344 r^{45} -
		6097558112152 r^{46} - 1749161262968 r^{47} \\
		&\quad{}+ 3118991463985 r^{48} +
		785756126384 r^{49} - 1329924471760 r^{50} \\
		&\quad{}- 293572096352 r^{51} +
		466929488512 r^{52} + 89618720832 r^{53} \\
		&\quad{}- 132649278016 r^{54} -
		21763809920 r^{55} + 29733666976 r^{56 }\\
		&\quad{}+ 4036919552 r^{57} -
		5063035648 r^{58} - 535680512 r^{59} + 615747584 r^{60} \\
		&\quad{}+ 45126656 r^{61} -
		47672320 r^{62} - 1806336 r^{63} + 1766656 r^{64})=0.
	\end{align*}
	Hence, the admissible radius \(	R_{\mathcal S_{L}}(\mathcal{SP}_{\pi/3})\) is determined by the smallest root of \(T_L\left(r,\pi/3\right)\) belonging to the interval \((0,1)\). A numerical computation yields
	\[
	R_{\mathcal S_{L}}(\mathcal{SP}_{\pi/3})=0.3029058506102136.\]
	
	Upon setting \(\alpha=\pi/6\), we obtain the following equation corresponding to the second resultant:
	\begin{align*}
		T_L\left(r, \frac{\pi}{6}\right)& =73691306573577785966592( 256 + 768 \sqrt{3} r - 13120 r^2 - 83712 \sqrt{3} r^3 \\
		&\quad{}- 621696 r^4 +
		2056896 \sqrt{3} r^5 + 52365312 r^6 + 70190400 \sqrt{3} r^7  \\
		&\quad{}-
		838288464 r^8 - 4209523104 \sqrt{3} r^9- 12862248960 r^{10}\\
		&\quad{} +
		72371365920 \sqrt{3} r^{11} + 682638816144 r^{12} -
		271110633072 \sqrt{3} r^{13}\\
		&\quad{} - 11730422241900 r^{14} -
		8811135576984 \sqrt{3} r^{15} + 85981601149404 r^{16}  \\
		&\quad{}+
		143822027223396 \sqrt{3} r^{17}- 399691348650176 r^{18}  \\
		&\quad{}-
		1348105573503300 \sqrt{3} r^{19}+ 2431195720774457 r^{20}  \\
		&\quad{}+
		10948023510786228 \sqrt{3} r^{21}- 17442025683355932 r^{22}  \\
		&\quad{}-
		71718616046068488 \sqrt{3} r^{23}+ 101142080150082696 r^{24} \\
		&\quad{} +
		355993640160859320 \sqrt{3} r^{25}- 452739146326598568 r^{26}  \\
		&\quad{}-
		1345962125949079200 \sqrt{3} r^{27}+ 1618479913080232092 r^{28}  \\
		&\quad{}+
		3964959050114141496 \sqrt{3} r^{29}- 4730747178766217664 r^{30} \\
		&\quad{}-
		9273707902711850952 \sqrt{3} r^{31} + 11390162645917708812 r^{32}  \\
		&\quad{}+
		17412804386427268128 \sqrt{3} r^{33}- 22462656802607200056 r^{34}  \\
		&\quad{}-
		26294455286153100648 \sqrt{3} r^{35}+ 35687980921372761454 r^{36}  \\
		&\quad{}+
		31658157909161846280 \sqrt{3} r^{37}- 44109380393117233024 r^{38}  \\
		&\quad{}-
		29694266066518809336 \sqrt{3} r^{39}+ 38954981495615109300 r^{40} \\
		&\quad{}+
		20698483340291180040 \sqrt{3} r^{41 }- 17413503911924906160 r^{42} \\
		&\quad{}-
		9728695344368215728 \sqrt{3} r^{43} - 11422208336219975646 r^{44}  \\
		&\quad{}+
		2410117390876788720 \sqrt{3} r^{45}+ 31086728645650872372 r^{46}  \\
		&\quad{}+
		93817862039551392 \sqrt{3} r^{47}- 31064175530648228304 r^{48} \\
		&\quad{} -
		907270828642920492 \sqrt{3} r^{49}+ 15635448581393593440 r^{50} \\
		&\quad{}+
		2472616315377403740 \sqrt{3} r^{51} + 1352818005013233909 r^{52}  \\
		&\quad{}-
		3652128834805814436 \sqrt{3} r^{53}- 9156119374776746668 r^{54}  \\
		&\quad{}+
		2619831506153199288 \sqrt{3} r^{55}+ 7250506627534705816 r^{56}  \\
		&\quad{}-
		222273670048047960 \sqrt{3} r^{57}- 1919226989378269968 r^{58}  \\
		&\quad{}-
		1197600067126849248 \sqrt{3} r^{59}- 1249849289976346530 r^{60}  \\
		&\quad{}+
		981905304347140056 \sqrt{3} r^{61}+ 1398524756112008268 r^{62}  \\
		&\quad{}-
		228619297077730224 \sqrt{3} r^{63}- 480701380921696968 r^{64}  \\
		&\quad{}-
		138651780359976228 \sqrt{3} r^{65}- 30464541699869496 r^{66}  \\
		&\quad{}+
		145434648985833300\sqrt{3} r^{67}+ 113836321292119785 r^{68}  \\
		&\quad{}-
		48425175732259644 \sqrt{3} r^{69}- 51942175490751612 r^{70}  \\
		&\quad{}-
		2201456572305912 \sqrt{3} r^{71}+ 8563648118607680 r^{72}  \\
		&\quad{}+
		9721579292373792 \sqrt{3} r^{73}+ 3164681852291056 r^{74}  \\
		&\quad{}-
		4569742295944320 \sqrt{3} r^{75}- 2403069304222560 r^{76}  \\
		&\quad{}+
		894267127801728 \sqrt{3} r^{77}+ 731559637830528 r^{78} \\
		&\quad{}+
		144461681750784 \sqrt{3} r^{79 }- 143779738555392 r^{80} \\
		&\quad{}-
		164028775056384 \sqrt{3} r^{81 }+ 24738693388800 r^{82 }  \\
		&\quad{}+
		57845912583168 \sqrt{3} r^{83}- 4707068621568 r^{84} \\
		&\quad{}-
		12000793721856 \sqrt{3} r^{85} + 515951576064 r^{86} +
		1500286912512 \sqrt{3} r^{87}  \\
		&\quad{}+ 28611723264 r^{88}-
		98860769280 \sqrt{3}r^{89} - 10276605952 r^{90} \\
		&\quad{}+
		2497609728 \sqrt{3} r^{91}+ 452263936 r^{92})\,r^{64}=0.
	\end{align*}
	Consequently, the admissible radius \(	R_{\mathcal S_{L}}(\mathcal{SP}_{\pi/6})\) is characterized as the smallest zero of \(T_L\left(r,\pi/6\right)\) in the interval \((0,1)\). Numerical value is given by
	\[
	R_{\mathcal S_{L}}(\mathcal{SP}_{\pi/6})=0.19459383342676778.
	\]

	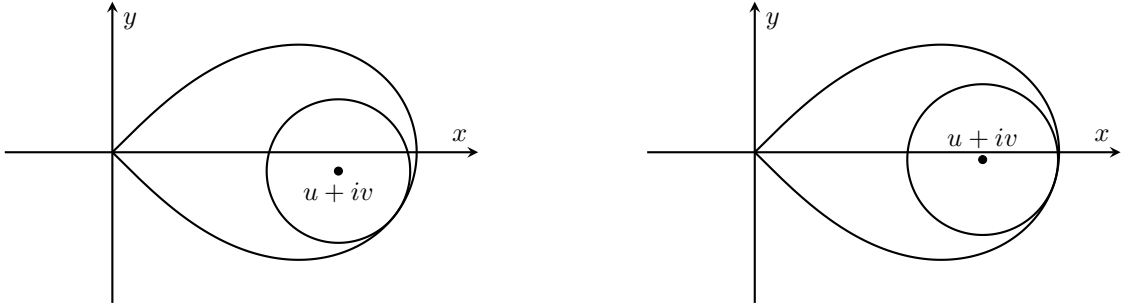
\begin{figure}[!htbp]
		\centering
		\begin{minipage}{0.48\textwidth}
			\centering
			\begin{tikzpicture}
				\begin{axis}
					[
					width=\linewidth,
					axis line style={thick},
					axis lines=middle,
					xmin=-0.5, xmax=1.7,
					ymin=-0.7, ymax=0.7,
					axis equal image,
					samples=200,
					xtick=\empty,
					ytick=\empty,
					xlabel={$x$},
					ylabel={$y$},
					label style={font=\footnotesize},
					]
					
					% Right loop of the lemniscate
					\addplot[black, thick, samples=200, domain=-45:45]
					({sqrt(2*cos(2*x))*cos(x)},
					{sqrt(2*cos(2*x))*sin(x)})
					-- (axis cs:0,0) -- cycle;
					
					% Circle
					\addplot[black, thick, domain=0:360, samples=300]
					({1.0505104055944359 + 0.33350564535271254*cos(x)},
					{-0.08748658880047408 + 0.33350564535271254*sin(x)});
					
					% Centre
					\addplot[only marks, mark=*, mark size=1.5pt]
					coordinates {(1.0505104055944359,-0.08748658880047408)};
					
					\node[below, font=\footnotesize]
					at (axis cs:1.0505104055944359,-0.08748658880047408) {$u+iv$};
				\end{axis}
			\end{tikzpicture}
		\end{minipage}\hfill
		\begin{minipage}{0.48\textwidth}
			\centering
			\begin{tikzpicture}
				\begin{axis}
					[
					width=\linewidth,
					axis line style={thick},
					axis lines=middle,
					xmin=-0.5, xmax=1.7,
					ymin=-0.7, ymax=0.7,
					axis equal image,
					samples=200,
					xtick=\empty,
					ytick=\empty,
					xlabel={$x$},
					ylabel={$y$},
					label style={font=\footnotesize},
					]
					
					% Right loop of the lemniscate
					\addplot[black, thick, samples=200, domain=-45:45]
					({sqrt(2*cos(2*x))*cos(x)},
					{sqrt(2*cos(2*x))*sin(x)})
					-- (axis cs:0,0) -- cycle;
					
					% Circle
					\addplot[black, thick, domain=0:360, samples=300]
					({1.0590356279677469 + 0.35031157050291983*cos(x)},
					{-0.0340842356989573 + 0.35031157050291983*sin(x)});
					
					% Centre
					\addplot[only marks, mark=*, mark size=1.5pt]
					coordinates {(1.0590356279677469,-0.0340842356989573)};
					
					\node[above, font=\footnotesize]
					at (axis cs:1.0590356279677469,-0.0340842356989573) {$u+iv$};
				\end{axis}
			\end{tikzpicture}
		\end{minipage}
		\caption{The disk with centre \(C=u+iv\) tangent to the right loop of the Bernoulli lemniscate for \(\alpha=\pi/3,\pi/6\).}
		\label{fig:lempi6}
	\end{figure}
	\noindent Specializing to \(\alpha=\pi/4\), we obtain the following resultant equation that determines the admissible values of the radius

	\begin{align*}
		T_L\left(r,\frac{\pi}{4}\right)& =\, 590295810358705651712\,r^{64}(2+6\sqrt{2}\,r-99r^{2}-526\sqrt{2}\,r^{3}-742r^{4}
		\\
		&\quad{}+14512\sqrt{2}\,r^{5}+143574r^{6}+18102\sqrt{2}\,r^{7}  -3205706r^{8}\\
		&\quad{}-8304698\sqrt{2}\,r^{9}	+11776119r^{10}+167664852\sqrt{2}\,r^{11}
		+635121140r^{12}
		\\
		&\quad{}-1197225056\sqrt{2}\,r^{13}-12099325416r^{14}
		-3986464220\sqrt{2}\,r^{15}
		\\
		&\quad{}+79833272086r^{16} +117052417434\sqrt{2}\,r^{17}
		-73426998993r^{18}
		\\
		&\quad{}-836572897014\sqrt{2}\,r^{19}-2141620126646r^{20} +3804319161296\sqrt{2}\,r^{21}
		\\
		&\quad{}+14980618622946r^{22}-18047353115194\sqrt{2}\,r^{23}
		\\
		&\quad{}	-37477144881238r^{24}+113661132229834\sqrt{2}\,r^{25}
		\\
		&\quad{}-111348678078891r^{26}-711172771081744\sqrt{2}\,r^{27}
		\\
		&\quad{}+1551169222663344r^{28} +3650336430201984\sqrt{2}\,r^{29}
		\\
		&\quad{}	-8534252157438784r^{30}-15043115417743888\sqrt{2}\,r^{31}
		\\
		&\quad{} +32849180932414388r^{32}+50758758134869884\sqrt{2}\,r^{33}
		\\
		&\quad{}-99500011190029422r^{34}-143269609601104316\sqrt{2}\,r^{35}
		\\
		&\quad{} +248949409621455028r^{36}+344047214423548128\sqrt{2}\,r^{37}
		\\
		&\quad{}-527880949197149796r^{38}-711861511327608148\sqrt{2}\,r^{39}
		\\
		&\quad{} +963165597736829276r^{40}+1281041965955180476\sqrt{2}\,r^{41}
		\\
		&\quad{}	-1526157449640276906r^{42}-2019032570840395528\sqrt{2}\,r^{43}
		\\
		&\quad{}	+2110686448098523512r^{44} +2801191318159362240\sqrt{2}\,r^{45}
		\\
		&\quad{}	-2552211455403146576r^{46}-3433195628961639656\sqrt{2}\,r^{47}
		\\
		&\quad{} +2694612506586862924r^{48} +3725245513969314996\sqrt{2}\,r^{49}
		\\
		&\quad{}-2472937068870364978r^{50}-3581810030001743932\sqrt{2}\,r^{51}
		\\
		&\quad{} +1956458854854502500r^{52}+3050751968075697568\sqrt{2}\,r^{53}
		\\
		&\quad{}	-1316000148561042172r^{54}-2298545267783297668\sqrt{2}\,r^{55}
		\\
		&\quad{} 	+734656463025627220r^{56}+1528388421153195348\sqrt{2}\,r^{57}
		\\
		&\quad{}	-324099737205429350r^{58}-894328665874723600\sqrt{2}\,r^{59}
		\\
		&\quad{}	+98526467249929520r^{60} +459163805151265408\sqrt{2}\,r^{61}
		\\
		&\quad{}	-7131276037862144r^{62}-206363358398172816\sqrt{2}\,r^{63}
		\\
		&\quad{}-14544233721760278r^{64} +81106840594679710\sqrt{2}\,r^{65}
		\\
		&\quad{}+11911356319188961r^{66}-27901785551864886\sqrt{2}\,r^{67}
		\\
		&\quad{} -5947344799771022r^{68}+8424456145397616\sqrt{2}\,r^{69}
		\\
		&\quad{}
		+2270983675780718r^{70}-2239956401360290\sqrt{2}\,r^{71}
		\\
		&\quad{}
		-709953539660274r^{72 }+525404907733534\sqrt{2}\,r^{73} \\
		&\quad{}
		+187219360578339r^{74}
		-108562290601132\sqrt{2}\,r^{75} \\
		&\quad{}
		-42164572222796r^{76}  +19672190188704\sqrt{2}\,r^{77} \\
		&\quad{}
		+8134206465400r^{78}
		-3101458834588\sqrt{2}\,r^{79}
		\\
		&\quad{}-1341125870082r^{80}+418739908242\sqrt{2}\,r^{81} 
		+186859075315r^{82}
		\\
		&\quad{}-47119538382\sqrt{2}\,r^{83}-21465969902r^{84}
		+4235322768\sqrt{2}\,r^{85}
		\\
		&\quad{}+1961504314r^{86}-275816194\sqrt{2}\,r^{87}
		\\
		&\quad{}-131295518r^{88}+9618498\sqrt{2}\,r^{89}
		+4809249r^{90})=0
	\end{align*}
	Among the roots of \(T_L(r,\pi/4)=0\) lying in the interval \((0,1)\), the smallest one determines the admissible radius. Consequently,
	\[
	R_{\mathcal S_{L}}(\mathcal{SP}_{\pi/4})=0.2308777657345702. \qedhere
	\]
	
	%		\begin{figure}[h]
		%			\centering
		%			
		%			\begin{tikzpicture}
			%				\begin{axis}
				%					[
				%					axis lines=middle,
				%					xmin=-0.5, xmax=2,
				%					ymin=-1, ymax=1,
				%					axis equal image,
				%					samples=200,
				%					xlabel={$x$},
				%					ylabel={$y$},
				%					]
				%					% Right loop
				%					\addplot[
				%					black,
				%					thick,
				%					samples=300,
				%					domain=-45:45
				%					]
				%					(
				%					{sqrt(2*cos(2*x))*cos(x)},
				%					{sqrt(2*cos(2*x))*sin(x)}
				%					)
				%					-- (axis cs:0,0) -- cycle;
				%					
				%					
				%					% Circle
				%					\addplot[
				%					black,
				%					thick,
				%					domain=0:360,
				%					samples=400
				%					]
				%					(
				%					{1.056305903128773 + 0.34489493430875345*cos(x)},
				%					{-0.056305903128772934+ 0.34489493430875345*sin(x)}
				%					);
				%					
				%					% Centre of the circle
				%					\addplot[
				%					only marks,
				%					mark=*,
				%					mark size=2pt
				%					]
				%					coordinates {(1.056305903128773,-0.056305903128772934)};
				%					
				%					% Label
				%					\node[above]
				%					at (axis cs:1.056305903128773,-0.056305903128772934)
				%					{$u+iv$};
				%					
				%				\end{axis}
			%			\end{tikzpicture}
		%			
		%			\caption{The disk with centre \(C=u+iv\) tangent to the right loop of the Bernoulli lemniscate for \(\alpha=\pi/4\).}
		%			\label{fig:lempi4}
		%			
		%		\end{figure}
	
	\end{proof}
	\noindent The radius when \(\alpha=0\) has already been discussed \cite{MR3902942}. The radius of starlikness using tangency approach is proved in the following.
	
	\begin{corollary}
	The radius of lemniscate starlikeness of the class of starlike functions is given by \(	R_{\mathcal S_{L}}(\mathcal{SP}_{0})=3 - 2 \sqrt2\). This radius is sharp.
	\end{corollary}
	\begin{proof}%\begin{figure}[h]
	%			\centering
	%			
	%			\begin{tikzpicture}
		%				\begin{axis}
			%					[
			%					axis lines=middle,
			%					xmin=-0.5, xmax=2,
			%					ymin=-1, ymax=1,
			%					axis equal image,
			%					samples=200,
			%					xlabel={$x$},
			%					ylabel={$y$},
			%					]
			%					% Right loop of the lemniscate
			%					\addplot[
			%					black,
			%					thick,
			%					samples=300,
			%					domain=-45:45
			%					]
			%					(
			%					{sqrt(2*cos(2*x))*cos(x)},
			%					{sqrt(2*cos(2*x))*sin(x)}
			%					)
			%					-- (axis cs:0,0) -- cycle;
			%					
			%					
			%					% Circle
			%					\addplot[
			%					black,
			%					thick,
			%					domain=0:360,
			%					samples=400
			%					]
			%					(
			%					{1.060660171779821 + 0.35355339059327373*cos(x)},
			%					{0 + 0.35355339059327373*sin(x)}
			%					);
			%					
			%					% Centre
			%					\addplot[
			%					only marks,
			%					mark=*,
			%					mark size=2pt
			%					]
			%					coordinates {(1.060660171779821,0)};
			%					
			%					\node[above]
			%					at (axis cs:1.060660171779821,0)
			%					{$u+iv$};
			%					
			%				\end{axis}
		%			\end{tikzpicture}
	%			
	%			\caption{The disk with centre \(C=u+iv\) tangent to the right loop of the Bernoulli lemniscate for \(\alpha=0\).}	
	%			\label{fig:lemniscate-circle}
	%			
	%		\end{figure}
	When \(\alpha=0,\) we have \[u(r)=\frac{1+r^2}{1-r^2},
	\qquad
	v(r)=0,
	\qquad
	R(r)=\frac{2r}{1-r^2}.\]
	Upon substituting the above expression into \eqref{p1cond1lem} and \eqref{p1cond2lem}, and the identity \(s^{2}=1-c^{2}\), we arrive at the following equations after eliminating the denominators.
	\begin{align} \notag
	F_1:& = -1 + 20 r^2 + 64 c r^3 + 26 r^4 + 64 c^2 r^4 + 64 c r^5 + 20 r^6 - r^8\\ \label{p1alphazero2}
	F_2:& = -64 r^3 (1 + 2 c r + r^2) s.\end{align}
	For a simultaneous solution to exist, both equations must vanish. In particular, equation \eqref{p1alphazero2} vanishes if and only if
	\(s=0\)
	or
	\(1+2cr+r^2=0\).
	
	\noindent If \(1 + 2 c r + r^2=0\), then \(c=(-1-r^2)/(2r)\) which is true only if \[\left|\frac{-1-r^2}{2r} \right|\le 1.\]
	Since \(0<r<1\), this is invalid.
	
	\noindent If \(s=0\), then \(c=\pm 1\). \(c=-1\) produces \[F_1=-1 + 20 r^2 - 64 r^3 + 26 r^4 + 64 r^4 - 64 r^5 + 20 r^6 - r^8. \] A numerical examination shows that \(F_1\) has no zero in the interval \((0,1)\).
	\(c=1\) produces \[F_1= - 1 + 20 r^2 + 64 r^3 + 26 r^4 + 64 r^4 + 64 r^5 + 20 r^6 - r^8.\]
	The zero \(	R_{\mathcal S_{L}}(\mathcal{SP}_{0})\) of \(F_1\) in the interval \((0,1)\) is the radius  of lemniscate starlikeness of the class \(\mathcal{SP}_0\).  A numerical computation yields \[	R_{\mathcal S_{L}}(\mathcal{SP}_{0})=3 - 2 \sqrt2\approx 0.1715728752538099.\qedhere\]

\end{proof}

\section{Radius of lune starlikeness}

We determine the radius of lune starlikeness for the class \(\mathcal{SP}_{\alpha}.\)  The procedure is same as done in the previous sections with the lune boundary equation \(\partial\Omega_{\rightmoon}=(x^2 + y^2)^2 - 6 x^2 - 2 y^2 + 1=0.\) Consider the equation
\begin{align} \label{p1lunethmeqn}
T(r,\alpha)=&4398046511104\, r^{24}\cos^{16}\alpha\,(1+r^2\cos 2\alpha)^2\sin^2 2\alpha\\ \notag
&\times(1+2\cos 2\alpha+\cos^2 2\alpha+\sin^2 2\alpha)(4-4r^2+r^4+4r^2\cos 2\alpha-2r^4\cos 2\alpha\\ \notag
&
+r^4\cos^2 2\alpha+r^4\sin^2 2\alpha)(2-2r^2+r^4-4r^2\cos^2\alpha+2r^2\cos 2\alpha
+r^4\cos^2 2\alpha\\ \notag
&+2r^2\sin 2\alpha-2r^4\sin 2\alpha+r^4\sin^2 2\alpha)^2(2-2r^2+r^4-4r^2\cos^2\alpha\\ \notag
&
+2r^2\cos 2\alpha+r^4\cos^2 2\alpha-2r^2\sin 2\alpha+2r^4\sin 2\alpha+r^4\sin^2 2\alpha)^2\\ \notag
&\times(4-4r^4+r^8-48r^2\cos^2\alpha+64r^4\cos^2\alpha
-24r^6\cos^2\alpha+16r^4\cos^4\alpha\\ \notag
&+16r^2\cos 2\alpha-8r^4\cos 2\alpha-8r^6\cos 2\alpha
+4r^8\cos 2\alpha-32r^4\cos^2\alpha\cos 2\alpha
\\ \notag
&+16r^6\cos^2\alpha\cos 2\alpha+20r^4\cos^2 2\alpha-16r^6\cos^2 2\alpha+2r^8\cos^2 2\alpha
\\ \notag
&-8r^6\cos^2\alpha\cos^2 2\alpha+8r^6\cos^3 2\alpha-4r^8\cos^3 2\alpha
+r^8\cos^4 2\alpha+4r^4\sin^2 2\alpha\\ \notag
&-2r^8\sin^2 2\alpha-8r^6\cos^2\alpha\sin^2 2\alpha
+8r^6\cos 2\alpha\sin^2 2\alpha-4r^8\cos 2\alpha\sin^2 2\alpha
\\ \notag
&+2r^8\cos^2 2\alpha\sin^2 2\alpha
+r^8\sin^4 2\alpha)(4-16r^2+20r^4-8r^6+r^8-16r^2\cos^2\alpha
\\ \notag
&+32r^4\cos^2\alpha-24r^6\cos^2\alpha+16r^4\cos^4\alpha
-8r^4\cos 2\alpha+16r^6\cos 2\alpha\\ \notag
&-4r^8\cos 2\alpha-16r^6\cos^2\alpha\cos 2\alpha
-4r^4\cos^2 2\alpha+8r^6\cos^2 2\alpha+2r^8\cos^2 2\alpha
\\ \notag
&-8r^6\cos^2\alpha\cos^2 2\alpha+4r^8\cos^3 2\alpha+r^8\cos^4 2\alpha
-4r^4\sin^2 2\alpha+8r^6\sin^2 2\alpha\\ \notag
&-2r^8\sin^2 2\alpha-8r^6\cos^2\alpha\sin^2 2\alpha
+4r^8\cos 2\alpha\sin^2 2\alpha+2r^8\cos^2 2\alpha\sin^2 2\alpha\\ \notag
&+r^8\sin^4 2\alpha)(8-48r^2+104r^4-104r^6+56r^8-18r^{10}
+3r^{12}-32r^2\cos^2\alpha\\ \notag
&+112r^4\cos^2\alpha-184r^6\cos^2\alpha
+128r^8\cos^2\alpha-32r^{10}\cos^2\alpha+32r^4\cos^4\alpha\\ \notag
&-32r^6\cos^4\alpha
+16r^8\cos^4\alpha-32r^4\cos 2\alpha+104r^6\cos 2\alpha-88r^8\cos 2\alpha
\\ \notag
&+22r^{10}\cos 2\alpha-16r^4\cos^2\alpha\cos 2\alpha
-32r^6\cos^2\alpha\cos 2\alpha
+16r^8\cos^2\alpha\cos 2\alpha\\ \notag
&+32r^6\cos^4\alpha\cos 2\alpha
-16r^4\cos^2 2\alpha+40r^6\cos^2 2\alpha+24r^8\cos^2 2\alpha\\ \notag
&-44r^{10}\cos^2 2\alpha
+11r^{12}\cos^2 2\alpha-40r^6\cos^2\alpha\cos^2 2\alpha
-16r^8\cos^2\alpha\cos^2 2\alpha\\ \notag
&+8r^{10}\cos^2\alpha\cos^2 2\alpha
+16r^8\cos^4\alpha\cos^2 2\alpha-8r^6\cos^3 2\alpha+56r^8\cos^3 2\alpha
\\ \notag
&-28r^{10}\cos^3 2\alpha-32r^8\cos^2\alpha\cos^3 2\alpha
+8r^8\cos^4 2\alpha+14r^{10}\cos^4 2\alpha\\ \notag
&-7r^{12}\cos^4 2\alpha-8r^{10}\cos^2\alpha\cos^4 2\alpha
+6r^{10}\cos^5 2\alpha+r^{12}\cos^6 2\alpha-8r^4\sin^2 2\alpha
\\ \notag
&+24r^6\sin^2 2\alpha-16r^8\sin^2 2\alpha+4r^{10}\sin^2 2\alpha
-r^{12}\sin^2 2\alpha+8r^6\cos^2\alpha\sin^2 2\alpha
\\ \notag
&+80r^8\cos^2\alpha\sin^2 2\alpha-40r^{10}\cos^2\alpha\sin^2 2\alpha
+16r^8\cos^4\alpha\sin^2 2\alpha\\ \notag
&-8r^6\cos 2\alpha\sin^2 2\alpha
+40r^8\cos 2\alpha\sin^2 2\alpha-20r^{10}\cos 2\alpha\sin^2 2\alpha
\\ \notag
&+96r^8\cos^2\alpha\cos 2\alpha\sin^2 2\alpha+8r^8\cos^2 2\alpha\sin^2 2\alpha
+20r^{10}\cos^2 2\alpha\sin^2 2\alpha\\ \notag
&-10r^{12}\cos^2 2\alpha\sin^2 2\alpha
+48r^{10}\cos^2\alpha\cos^2 2\alpha\sin^2 2\alpha+12r^{10}\cos^3 2\alpha\sin^2 2\alpha
\\  \notag
&+3r^{12}\cos^4 2\alpha\sin^2 2\alpha+6r^{10}\sin^4 2\alpha-3r^{12}\sin^4 2\alpha
-8r^{10}\cos^2\alpha\sin^4 2\alpha\\ \notag
&+6r^{10}\cos 2\alpha\sin^4 2\alpha
+3r^{12}\cos^2 2\alpha\sin^4 2\alpha
+r^{12}\sin^6 2\alpha)^2=0.
\end{align}

\begin{theorem}
Let \(0<|\alpha|<\pi/2.\) The radius of lune starlikeness for the class \(\mathcal{SP}_{\alpha}\) is the smallest admissible root in \((0,1)\) of \(T_{\rightmoon}(r,\alpha)=\operatorname{det}S(R_1,G)=0,\) where \(R_1, \text{ and } G\) are given by \eqref{p1lunethmeqn}.
\end{theorem} 
\begin{proof}
Let \(\Omega_{\rightmoon}=\{w\in \mathbb{C}: |w^2-1|<2|w|\}.\) Then the boundary of \(\Omega_{\rightmoon}\) is 
\(
|w^2-1|=2|w|.
\) 
Let \(w=x+iy,\) then the boundary of \(\Omega_{\rightmoon}\) will become \(|x^2-y^2+2ixy-1|=2|x+iy|,\) which produces   
\[\Phi(x,y):=(x^2 + y^2)^2 - 6 x^2 - 2 y^2 + 1=0.\]
We define the function \(F:(0,1) \times [0,2\pi] \rightarrow \mathbb{R}\) along the curve \(\gamma(t)\) in \(\eqref{p1gammaeqn}\) as 
\begin{equation} \label{p1luneeqn}
F(r,t):=\Phi(\gamma(t))=(x^2(r,t) + y^2(r,t)^2) - 6 x^2(r,t) - 2 y^2(r,t) + 1
\end{equation}
The sign of \(F(r,t)\) determine the position of the point \(\gamma(t)\) relative to the region \(\Omega_{\rightmoon}.\) The disk determined by the interior of \(\gamma(t)\) is contained in \(\Omega_{\rightmoon}\) precisely when \[
F(r,t)\le 0 \quad \text{ for all } t.
\]
The extremal situation  occurs when equality is attained for at least one value of \(t.\)   We want the disk \eqref{p1reqdisk} to lie completely inside \(\Omega_{\rightmoon}.\)  Therefore we have to prove the tangency condition for the boundary curves \(\Phi\) and \(\partial\Omega_{\rightmoon}.\)

\textbf{Tangency condition: }  At the extremal radius, the boundary circle is tangent to the lune defining \(\partial\Omega_{\rightmoon}.\)  As a consequence, there exists \(t=t_0\) such that 
\begin{align} \label{p1luneeqn1}
F(r,t)&=0,\\ \label{p1luneeqn2}
F_t(r,t)&=0. 
\end{align} 
Evaluating \eqref{p1luneeqn1} and \eqref{p1luneeqn2} using \eqref{p1luneeqn}, \eqref{p1gammaeqn}
and the identitly \(s^2=1-c^2, \) we get
\begin{align}\label{p1lunefeqn}
F(r,t)=&1 - 2 R^2 - 4 c^2 R^2 + R^4 - 12 c R u + 4 c R^3 u - 6 u^2 + 
2 R^2 u^2 + 4 c^2 R^2 u^2 \\ \notag
&+ 4 c R u^3 + u^4 - 2 v^2 + 
6 R^2 v^2 - 4 c^2 R^2 v^2 + 4 c R u v^2 + 2 u^2 v^2 + v^4 \\ \notag
&+ 
sR (-4 v + 4 R^2 v  + 8 c R u v+ 4 u^2 v + 4 v^3)=0\\ \label{p1lunefteqn}
F_t(r,t)=&4R(- c  v + c R^2 v - 2 R u v + 4 c^2 R u v + 
c  u^2 v +  c  v^3 ) \\ \notag
&+ 
4sR (2 c R + 3 u -  R^2 u - 2 c R u^2 -  u^3 + 
2 c R v^2 -  u v^2)=0.
\end{align}
From \eqref{p1lunefteqn},  on computing \(s\) we have 
\[
s=\frac{-(- c  v + c R^2 v - 2 R u v + 4 c^2 R u v + 
c  u^2 v +  c  v^3 )}{ (2 c R + 3 u -  R^2 u - 2 c R u^2 -  u^3 + 
2 c R v^2 -  u v^2)}.
\]
If the denominator vanishes, then it follows from \eqref{p1lunefteqn} that 
\[
F_t=4R(- c  v + c R^2 v - 2 R u v + 4 c^2 R u v + 
c  u^2 v +  c  v^3 )=0.
\]
On solving this equation together with \eqref{p1lunefeqn} we get solution \(r>	R_{\mathcal S_{\!\!\rightmoon}}(\mathcal{SP}_{\alpha}).\)  Since \(	R_{\mathcal S_{\!\!\rightmoon}}(\mathcal{SP}_{\alpha})\) is the smallest admissible radius at which tangency occurs, every such solution corresponds to a tangency that takes place after the disk has already ceased to be contained in \(\Omega_{\rightmoon}.\) Therefore, the case \(2 c R + 3 u -  R^2 u - 2 c R u^2 -  u^3 + 
2 c R v^2 -  u v^2=0\) should be neglected from further consideration for finding the admissible solution for the system.  If we substitute the value of \(s\) in the identity \(c^2+s^2=1\) and simplify, we get \begin{align*}
G(c)=& b_4c^4+b_3c^3+b_2c^2+b_1c+b_0
\end{align*}
where the coeffecients are given by
\begin{align}  \label{p1lunegeqn}
b_4=&64  R^4 - 128  R^4 u^2 + 64  R^4 u^4 + 128  R^4 v^2 + 
128  R^4 u^2 v^2 + 64 R^4 v^4,\\ \notag
b_3=&192  R^3 u - 64  R^5 u - 256  R^3 u^3 + 64  R^5 u^3 + 
64  R^3 u^5 + 64 R^5 u u^2 + 128  R^3 u^3 v^2 + 64  R^3 u v^4,\\ \notag
b_2=&-64  R^4 + 144  R^2 u^2 + 32  R^4 u^2 + 16  R^6 u^2 - 
96  R^2 u^4 - 32  R^4 u^4 + 16  R^2 u^6 + 16 R^2 v^2 \\ \notag
&- 
160  R^4 v^2  + 16  R^6 v^2- 128  R^2 u^2 v^2 - 64  R^4 u^2 v^2 + 
48  R^2 u^4 v^2 - 32  R^2 v^4 - 32  R^4 v^4 \\ \notag
&+ 48  R^2 u^2 v^4 + 
16  R^2 v^6, \\ \notag
b_1=&-192  R^3 u + 64  R^5 u + 256  R^3 u^3 - 64  R^5 u^3 - 
64 R^3 u^5 - 64  R^3 u v^2 - 64  R^3 u^3 v^2,\\ \notag
b_0=&-144 R^2 u^2 + 96 R^4 u^2 - 16 R^6 u^2 + 96 R^2 u^4 - 
32 R^4 u^4 - 16 R^2 u^6+ 96 R^2 u^2 v^2\\ \notag
&+ 32  R^4 u^2 v^2 - 
32  R^2 u^4 v^2 - 16  R^2 u^2  v^4.
\end{align}

\textbf{Resultant computation:} For eliminating \(s\) in equations \eqref{p1luneeqn1} and \eqref{p1luneeqn2}, we compute the resultant considerint them as a linear equation in \(s,\) we write
\[
F=g_0+g_1s, \qquad F_t=h_0+h_1s
\]
where 
\begin{align} \notag
g_0=&1 - 2 R^2 - 4 c^2 R^2 + R^4 - 12 c R u + 4 c R^3 u - 6 u^2 + 
2 R^2 u^2 + 4 c^2 R^2 u^2 \\ \notag
&+ 4 c R u^3 + u^4 - 2 v^2 + 
6 R^2 v^2 - 4 c^2 R^2 v^2 + 4 c R u v^2 + 2 u^2 v^2 + v^4 \\ \notag
g_1=&R (-4 v + 4 R^2 v  + 8 c R u v+ 4 u^2 v + 4 v^3)\\ \notag
h_0=&4R(- c  v + c R^2 v - 2 R u v + 4 c^2 R u v + 
c  u^2 v +  c  v^3 )\\ \notag
h_1=&4R (2 c R + 3 u -  R^2 u - 2 c R u^2 -  u^3 + 
2 c R v^2 -  u v^2)
\end{align}
The Resultant of the equations \(F=0 \text{ and } F_t=0\) vanishes if both of them admit a common solution in \(s\) and it is given by 
\[\operatorname{Res}_s(F,F_t)=\begin{vmatrix}
g_1&g_0\\
h_1&h_0
\end{vmatrix}=g_1h_0-g_0h_1=0.\]
On substituting the values of \(g_i \text{ and } h_i\) we get 
\begin{align*}
\operatorname{Res}_s(F,F_t) =&a_3c^3+a_2c^2+a_1c+a_0=R_1(c).
\end{align*} 
where \(a_i\)'s are given by 
\begin{align*}
a_3=&32 R^4 - 64 R^4 u^2 + 32 R^4 u^4 + 64 R^4 v^2 + 
64 R^4 u^2 v^2 + 32 R^4 v^4\\ \notag
a_2=&144 R^3 u - 48 R^5 u - 192 R^3 u^3  + 48 R^5 u^3 + 
48 R^3 u^5 + 48 R^5 u v^2 + 96 R^3 u^3 v^2 + 48 R^3 u v^4\\ 
a_1=&-8 R^2 + 16 R^4 - 8 R^6 + 200 R^2 u^2 - 128 R^4 u^2 + 
24 R^6 u^2 - 152 R^2 u^4 + 48 R^4 u^4 + 24 R^2 u^6 \\ \notag
&+ 
24 R^2 v^2  - 64 R^4 v^2 + 8 R^6 v^2 - 112 R^2 u^2 v^2 + 
32 R^4 u^2 v^2 + 56 R^2 u^4 v^2 - 24 R^2 v^4 \\ \notag
&- 16 R^4 v^4+ 
40 R^2 u^2 v^4 + 8 R^2 v^6\\ \notag
a_0=&-12 R u + 28 R^3 u - 20 R^5 u + 4 R^7 u + 76 R u^3 - 
56 R^3 u^3 + 12 R^5 u^3 - 36 R u^5 \\ \notag
&+ 12 R^3 u^5 + 4 R u^7 + 
28 R u v^2 - 56 R^3 u v^2 - 4 R^5 u v^2 - 56 R u^3 v^2 + 
8 R^3 u^3 v^2 + 12 R u^5 v^2 \\ \notag
&- 20 R u v^4 - 4 R^3 u v^4 + 
12 R u^3 v^4 + 4 R u v^6
\end{align*}
If the equations \(R_1(c)=0 \text{ and } G(c)=0\) have a common solution then their resultant with respect to \(c,\) \(\operatorname{Res}_c(R_1,G)\)  vanishes for some \(r=	R_{\mathcal S_{\!\!\rightmoon}}(\mathcal{SP}_{\alpha}).\)  The resultant \(\operatorname{Res}_c(R_1,G)\) is the determinant of the sylvester matrix of \(R_1 \text{ and } G.\) 
\begin{align*}
\operatorname{Res}_c(R_1,G)=&\begin{vmatrix}
	b_4&b_3&b_2&b_1&b_0&0&0\\
	0&b_4&b_3&b_2&b_1&b_0&0\\
	0&0&b_4&b_3&b_2&b_1&b_0\\
	a_3&a_2&a_1&a_0&0&0&0\\
	0&a_3&a_2&a_1&a_0&0&0\\
	0&0&a_3&a_2&a_1&a_0&0\\
	0&0&0&a_3&a_2&a_1&a_0
\end{vmatrix}=0.\\	
\end{align*}
On evaluating the determinant, we get
\begin{align*}
\operatorname{Res}_c(R_1,G)=&67108864 R^{16} u^2 v^2 (-1 + R^2 - u^2 - 2 v - v^2)^2 (-1 + R^2 - 
u^2 + 2 v - v^2)^2 \\
&(1 - 2 u + u^2 + v^2) (1 + 2 u + u^2 + 
v^2) (1 - 6 R^2 + R^4 + 4 u + 4 R^2 u + 2 u^2 \\
& - 2 R^2 u^2- 
4 u^3 + u^4 - 2 v^2 - 2 R^2 v^2 - 4 u v^2 + 2 u^2 v^2 + v^4)\\
& (1 - 
6 R^2 + R^4 - 4 u - 4 R^2 u + 2 u^2 - 2 R^2 u^2 + 4 u^3 + 
u^4 - 2 v^2 - 2 R^2 v^2\\
& + 4 u v^2 + 2 u^2 v^2 + v^4) (3 - 8 R^2 +
R^4 + 11 u^2 + 2 R^2 u^2 + R^4 u^2 - 7 u^4  \\
& - 2 R^2 u^4 + u^6-
v^2 - 10 R^2 v^2 + R^4 v^2 - 10 u^2 v^2 + 12 R^2 u^2 v^2 + 
3 u^4 v^2 - 3 v^4  \\
& - 2 R^2 v^4 + 3 u^2 v^4 + v^6)^2.
\end{align*}
On substituting the values of \(u, v \text{ and } R\) and canceling the denominator factors which are the powers of \((1-r)(1+r),\) we get the resultant equation to be \(T(r,\alpha)=0\) given by \eqref{p1lunethmeqn}.
Among the zeros of \(T(r,\alpha),\) we find the admissible solution for the system \eqref{p1luneeqn1}, \eqref{p1luneeqn2}, \eqref{p1lunegeqn} and the smallest admissible solution \(	R_{\mathcal S_{\!\!\rightmoon}}(\mathcal{SP}_{\alpha})\) is the radius of lune starlikeness for the class \(\mathcal{SP}_{\alpha}.\)
\end{proof}
For some particular values of \(\alpha,\) the lune starlikeness radii has been given in the following corollary. 
\begin{corollary}
Let $ 	R_{\mathcal S_{\!\!\rightmoon}}(\mathcal{SP}_{\alpha})$ be the radius of parabolic starlikeness for the class
$\mathcal{SP}_\alpha$. Then  \begin{align*}
R_{\mathcal S_{\!\!\rightmoon}}(\mathcal{SP}_{\pi/3}) &= 2 - \sqrt2,
& 	R_{\mathcal S_{\!\!\rightmoon}}(\mathcal{SP}_{\pi/4})&= 1/2\\
R_{\mathcal S_{\!\!\rightmoon}}(\mathcal{SP}_{\pi/6})&= -2 + \sqrt6,
& 	R_{\mathcal S_{\!\!\rightmoon}}(\mathcal{SP}_{\pi/8}) &= -\sqrt2 + \sqrt{2 + \sqrt2}.
\end{align*}These radii are sharp.
\end{corollary}
\begin{proof}
For \(\alpha=\pi/3,\) the final resultant equation \eqref{p1lunethmeqn} becomes 
\begin{align*}
T\left(r,\frac{\pi}{3}\right)=&50331648 r^{24} (1 - r^2/2)^2 (4 - 6 r^2 + 3 r^4) (2 - 4 r^2 + 
\sqrt3 r^2 + 2 r^4 - \sqrt3 r^4)^2  \\
&\times (2 - 4 r^2 - \sqrt3 r^2+ 
2 r^4 + \sqrt3 r^4)^2 (4 - 20 r^2 + 29 r^4 - 14 r^6 + r^8)^2  \\
&\times (8 - 
56 r^2 + 142 r^4 - 170 r^6 + 112 r^8 - 35 r^{10} + 2 r^{12})^2=0.
\end{align*}
On solving this equation and finding the admissible roots, we get that \(	R_{\mathcal S_{\!\!\rightmoon}}(\mathcal{SP}_{\pi/3}) = 2 - \sqrt2.\)

For \(\alpha=\pi/4,\) \eqref{p1lunethmeqn} becomes
\begin{align*}
T\left(r,\frac{\pi}{4}\right)=& 34359738368 r^{24} (2 - 2 r^2)^2 (4 - 4 r^2 + 2 r^4) (2 - 6 r^2 + 
4 r^4)^2 \\
& \times (4 - 24 r^2 + 36 r^4 - 16 r^6)^2 (8 - 64 r^2 + 160 r^4 - 
176 r^6 + 152 r^8 - 48 r^{10})^2=0.
\end{align*}
The admissible solution of this equation for the system   \eqref{p1luneeqn1}, \eqref{p1luneeqn2} and \eqref{p1lunegeqn} is given by is \(	R_{\mathcal S_{\!\!\rightmoon}}(\mathcal{SP}_{\pi/4})= 1/2.\) 

For \(\alpha=\pi/6,\) \eqref{p1lunethmeqn} becomes
\begin{align*}
T\left(r,\frac{\pi}{6}\right)=& 990677827584 r^{24} (1 + r^2/2)^2 (4 - 2 r^2 + r^4) (2 - 4 r^2 + 
\sqrt3 r^2 + 2 r^4 - \sqrt3 r^4)^2 \\
&\times (2 - 4 r^2 - \sqrt3 r^2 + 
2 r^4 + \sqrt3 r^4)^2 (4 - 28 r^2 + 45 r^4 - 22 r^6 + r^8)^2\\
&\times  (8 - 
72 r^2 + 174 r^4 - 190 r^6 + 216 r^8 - 57 r^{10} + 2 r^{12})^2=0.
\end{align*}
We choose the admissible solution for the system \eqref{p1luneeqn1}, \eqref{p1luneeqn2} and \eqref{p1lunegeqn} among the roots of \(T\left(r,\pi/6\right)=0.\)  Therefore, we have \(	R_{\mathcal S_{\!\!\rightmoon}}(\mathcal{SP}_{\pi/6})=-2 + \sqrt6.\)

Substituting  \(\alpha=\pi/8\) in \eqref{p1lunethmeqn} yields, 
\begin{align*}
T\left(r,\frac{\pi}{8}\right)=&2199023255552 (2 + \sqrt2) r^{24} (1 + r^2/\sqrt2)^2 (4 - 4 r^2 + 
2 \sqrt2 r^2 + 2 r^4 - \sqrt2 r^4) \\
&\times ((2+\sqrt{2})/4)^{8} \left(2 - 2 r^2 + 2 \sqrt{2} r^2 + 2 r^4 - \sqrt{2} r^4 - 
r^2 (2+\sqrt{2})\right)^2 \\
&\times \left(2 - 2 r^2 + 2 r^4 + \sqrt{2} r^4 - 
r^2 (2+\sqrt{2})\right)^2 \\
&\times \left(4 - 16 r^2 + 16 r^4 - 4 \sqrt{2} r^4 + 
8 \sqrt{2} r^6  + 2 r^8 - 4 r^2 (2+\sqrt{2}) + 
8 r^4 (2+\sqrt{2}) \right. \\
& \left. - 8 r^6 (2+\sqrt{2}) - 
2 \sqrt{2} r^6 (2+\sqrt{2})+ r^4 (2+\sqrt{2})^2 \right)\\
&\times \left(4 + 
8 \sqrt{2} r^2 + 8 r^4 - 4 \sqrt{2} r^4 - 8 r^6 + 2 r^8 - 
12 r^2 (2+\sqrt{2})  + 16 r^4 (2+\sqrt{2}) \right. \\
&\left.  - 
4 \sqrt{2} r^4 (2+\sqrt{2}) - 8 r^6 (2+\sqrt{2})+ 
2 \sqrt{2} r^6 (2+\sqrt{2}) + r^4 (2+\sqrt{2})^2\right) \\
& \left(8 - 
48 r^2 + 92 r^4 - 16 \sqrt{2} r^4 - 72 r^6 + 48 \sqrt{2} r^6 + 
64 r^8 - 20 \sqrt{2} r^8 - 28 r^{10 } \right. \\
&  + 2 \sqrt{2} r^{10} + 4 r^{12} - 
8 r^2 (2+\sqrt{2})+ 28 r^4 (2+\sqrt{2}) - 
2 \sqrt{2} r^4 (2+\sqrt{2}) \\
& - 50 r^6 (2+\sqrt{2}) - 
4 \sqrt{2} r^6 (2+\sqrt{2}) + 40 r^8 (2+\sqrt{2}) + 
6 \sqrt{2} r^8 (2+\sqrt{2}) \\
&- 10 r^{10} (2+\sqrt{2})+ 
2 r^4 (2+\sqrt{2})^2 - 2 r^6 (2+\sqrt{2})^2 + 
\sqrt{2} r^6 (2+\sqrt{2})^2 \\
&\left. + 2 r^8 (2+\sqrt{2})/4\right)^2=0.
\end{align*}
Among the roots of \(T\left(r,\pi/8\right)=0,\) the admissible solution for the system \eqref{p1luneeqn1}, \eqref{p1luneeqn2} and \eqref{p1lunegeqn} is \(	R_{\mathcal S_{\!\!\rightmoon}}(\mathcal{SP}_{\pi/8})=-\sqrt2 + \sqrt{2 + \sqrt2}.\)
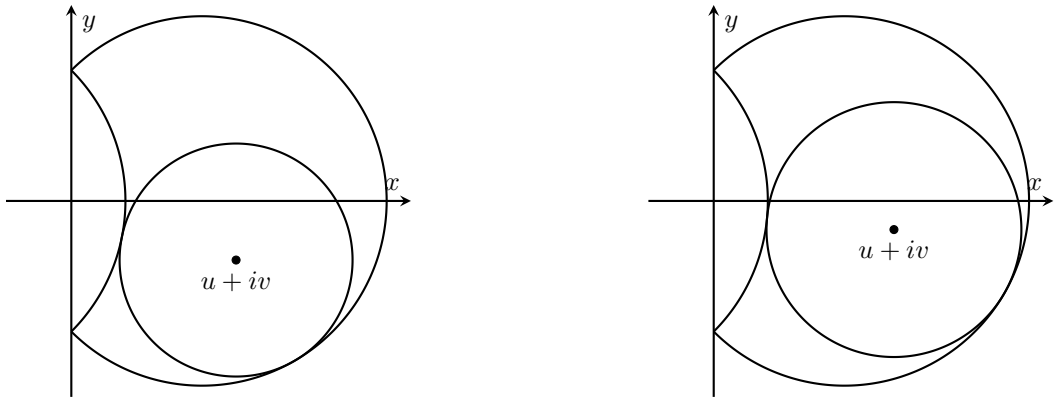
\begin{figure}[!htbp]
\centering
\begin{minipage}{0.48\textwidth}
	\centering
	\begin{tikzpicture}
		% Parameters
		\pgfmathsetmacro{\rzero}{2-sqrt(2)}
		\pgfmathsetmacro{\alphadeg}{60}
		\pgfmathsetmacro{\ucenter}{(1+\rzero^2*cos(2*\alphadeg))/(1-\rzero^2)}
		\pgfmathsetmacro{\vcenter}{-\rzero^2*sin(2*\alphadeg)/(1-\rzero^2)}
		\pgfmathsetmacro{\rad}{2*\rzero*cos(\alphadeg)/(1-\rzero^2)}
		
		\begin{axis}
			[
			width=\linewidth,
			axis lines=middle,
			axis line style={thick},
			xmin=-0.5, xmax=2.6,
			ymin=-1.5, ymax=1.5,
			axis equal image,
			xtick=\empty,
			ytick=\empty,
			xlabel={$x$},
			ylabel={$y$},
			label style={font=\footnotesize},
			]
			
			% Outer boundary
			\addplot[black, thick, domain=-90:90, samples=300]
			({sqrt((2+4*cos(x)^2+sqrt((2+4*cos(x)^2)^2-4))/2)*cos(x)},
			{sqrt((2+4*cos(x)^2+sqrt((2+4*cos(x)^2)^2-4))/2)*sin(x)});
			
			% Inner boundary
			\addplot[black, thick, domain=-90:90, samples=300]
			({sqrt((2+4*cos(x)^2-sqrt((2+4*cos(x)^2)^2-4))/2)*cos(x)},
			{sqrt((2+4*cos(x)^2-sqrt((2+4*cos(x)^2)^2-4))/2)*sin(x)});
			
			% Circle
			\addplot[black, thick, domain=0:360, samples=300]
			({\ucenter+\rad*cos(x)}, {\vcenter+\rad*sin(x)});
			
			% Centre
			\addplot[only marks, mark=*, mark size=1.5pt]
			coordinates {(\ucenter,\vcenter)};
			
			\node[below, font=\footnotesize]
			at (axis cs:\ucenter,\vcenter) {$u+iv$};
		\end{axis}
	\end{tikzpicture}
\end{minipage}\hfill
\begin{minipage}{0.48\textwidth}
	\centering
	\begin{tikzpicture}
		% Parameters
		\pgfmathsetmacro{\rzero}{0.4494897427831781}
		\pgfmathsetmacro{\alphadeg}{30}
		\pgfmathsetmacro{\ucenter}{(1+\rzero^2*cos(2*\alphadeg))/(1-\rzero^2)}
		\pgfmathsetmacro{\vcenter}{-\rzero^2*sin(2*\alphadeg)/(1-\rzero^2)}
		\pgfmathsetmacro{\rad}{2*\rzero*cos(\alphadeg)/(1-\rzero^2)}
		
		\begin{axis}
			[
			width=\linewidth,
			axis lines=middle,
			axis line style={thick},
			xmin=-0.5, xmax=2.6,
			ymin=-1.5, ymax=1.5,
			axis equal image,
			xtick=\empty,
			ytick=\empty,
			xlabel={$x$},
			ylabel={$y$},
			label style={font=\footnotesize},
			]
			
			% Outer boundary
			\addplot[black, thick, domain=-90:90, samples=300]
			({sqrt((2+4*cos(x)^2+sqrt((2+4*cos(x)^2)^2-4))/2)*cos(x)},
			{sqrt((2+4*cos(x)^2+sqrt((2+4*cos(x)^2)^2-4))/2)*sin(x)});
			
			% Inner boundary
			\addplot[black, thick, domain=-90:90, samples=300]
			({sqrt((2+4*cos(x)^2-sqrt((2+4*cos(x)^2)^2-4))/2)*cos(x)},
			{sqrt((2+4*cos(x)^2-sqrt((2+4*cos(x)^2)^2-4))/2)*sin(x)});
			
			% Circle
			\addplot[black, thick, domain=0:360, samples=300]
			({\ucenter+\rad*cos(x)}, {\vcenter+\rad*sin(x)});
			
			% Centre
			\addplot[only marks, mark=*, mark size=1.5pt]
			coordinates {(\ucenter,\vcenter)};
			
			\node[below, font=\footnotesize]
			at (axis cs:\ucenter,\vcenter) {$u+iv$};
		\end{axis}
	\end{tikzpicture}
\end{minipage}
\caption{The disk with centre \(C=u+iv\) tangent to the right moon for \(\alpha=\pi/3,\pi/6\).}
\label{fig:lunepi6}
\end{figure}
\end{proof}

For the case \(\alpha=0,\) the class \(\mathcal{SP}_{\alpha}\) is reduced to the class \(\mathcal{ST}\) for which the starlikeness radius has been already found in \cite{MR4861252}.  The tangency approach is discussed in the following corollary.
\begin{corollary}
The radius of lune starlikeness for the class of starlike functions is \\ \(	R_{\mathcal S_{\!\!\rightmoon}}(\mathcal{SP}_{0})=\sqrt{2}-1.\) This radius is sharp. 
\end{corollary}
\begin{proof}
For \(\alpha=0.\) the equation \eqref{p1reqdisk} gives
\[ u(r)=\frac{1+r^2}{1-r^2}, \qquad v(r)=0, \qquad R(r)=\frac{2r}{1-r^2}.\] Therefore, \eqref{p1gammaeqn} yields \(x=u+Rc, \quad y=Rs.\)
In particular, we have to substitute \(v=0\) in \eqref{p1lunefeqn} and \eqref{p1lunefteqn}.  We have,
\begin{align} \label{p1lunealpha0F}
F(r,t)=&1 - 2 R^2 - 4 c^2 R^2 + R^4 - 12 c R u + 4 c R^3 u - 6 u^2 + 
2 R^2 u^2 + 4 c^2 R^2 u^2 \\ \notag
&+ 4 c R u^3 + u^4 =0\\ \notag
F_t(r,t)=& 
4sR (2 c R + 3 u -  R^2 u - 2 c R u^2 -  u^3)=0.
\end{align}
\(F_t=0\) if either \(s=0\) or \(2 c R + 3 u -  R^2 u - 2 c R u^2 -  u^3=0.\) Let us analyze both the cases.   

If \(G:=2 c R + 3 u -  R^2 u - 2 c R u^2 -  u^3=0,\)  After substituting the values of \(u \text{ and } R\) in \(G=0, \) we get 
\begin{equation} \label{p1lunealpha0c}
c=\frac{1 - 5 r^2 - 5 r^4 + r^6}{8r^3}
\end{equation}
Also, from \eqref{p1lunealpha0F} we have
\[
F(r,t)=\frac{4 (1 + 4 c r + 4 r^2 - r^4) (-1 + r^2 (4 + 4 c r + r^2))}{(-1 + r^2)^4}=0.
\]
On substituting the value of \(c\) from \eqref{p1lunealpha0c} in the above equation we get 
\[
F(r)=-\frac{(-1 + r^2)^2}{r^2}
\]
Since, \(0<r<1,\) we can see that \(F \ne 0.\)  Therefore, \(G=0\) is not a possibility.

If \(s=0,\) from the identity \(c^2+s^2=1\) we get \(c=\pm 1.\) Substituting these values in \eqref{p1lunealpha0F}, we get \(r=0.4142135623730951\) for both \(c=1 \text{ and } c=-1.\)  Therefore, \[	R_{\mathcal S_{\!\!\rightmoon}}(\mathcal{SP}_{0})=0.4142135623730951\approx \sqrt{2}-1. \qedhere\]
\end{proof}

\bigskip
\parskip6pt \parindent0in
\textbf{Declarations}

\textit{Acknowledgments.} The first author is supported by a junior research fellowship from UGC, New Delhi, India. 

\noindent\textit{Author Contributions.} Both authors contributed equally   and approved the final manuscript.

\textit{Conflict of Interest.} No conflict of interest.

\textit{Competing Interests.} The authors declare no competing interests.

\textit{Data Availability.} No datasets were generated or analysed during the current study.

\textit{Funding.} No funding.

\textit{Code Availability.} Some   symbolic computations were carried out in \texttt{Mathematica} and code is available on request.

\textit{Materials Availability.} No datasets were generated or analysed during the current study.

\textit{Use of Large Language Models, AI and Machine Learning Tools.} None declared.

\textit{Ethics Approval and Consent to Participate.} Not applicable.

\textit{Consent for Publication.} Not applicable.

\bibliographystyle{plain}
\bibliography{references2}

\end{document}